\documentclass[11pt,a4paper]{article}

\usepackage{amsmath}
\usepackage{latexsym}
\usepackage[english]{babel}
\usepackage{amsfonts}
\usepackage[psamsfonts]{eucal}
\usepackage[T1]{fontenc}
\usepackage[latin1]{inputenc}
\usepackage{amsthm}
\usepackage{amssymb}
\usepackage{graphicx}
\usepackage{version}
\usepackage{fancyhdr}
\usepackage{caption}

\def \A {{\mathcal{A}}}
\def \F {{\mathcal{F}}}
\def \G {{\mathcal{G}}}
\def \X {{\mathcal{X}}}
\def \C {{\mathcal{C}}}

\def \P {{\mathbb{P}}}

\def \R {{\mathbb{R}}}
\def \D {{\mathcal{D}}}

\def \T {{\mathbb{T}}}

\newtheorem{theorem}{Theorem}[section]
\newtheorem{lemma}[theorem]{Lemma}
\newtheorem{definition}[theorem]{Definition}
\newtheorem{remark}[theorem]{Remark}
\newtheorem{proposition}[theorem]{Proposition}
\newtheorem{example}[theorem]{Example}

\newtheorem{ass}[theorem]{Assumption}
\numberwithin{equation}{section}

\newcommand{\esp}[2][\mathbb E] {#1\left[#2\right]}

\newcommand{\ud}{\mathrm{d}}
\newcommand{\ds}{\displaystyle}
\title{
Optimal consumption policies in illiquid markets\thanks{This work
is supported partly  by the Europlace Institute of Finance.} }
\author{Alessandra Cretarola$^{1)}$, Fausto Gozzi$^{1)}$, Huy\^en
Pham$^{2),3)}$, Peter Tankov$^{2)}$}

\date{}

\begin{document}
\maketitle {\noindent \small
\begin{tabular}{llll}
$^{1)}$ & Dipartimento di Scienze Economiche                &  $^{2)}$ & Laboratoire de Probabilités et \\
        & ed Aziendali - Facoltà di Economia,               &          & Modelèles Aléatoires,\\
        & Università LUISS Guido Carli,                 &          & CNRS, UMR 7599\\
        & viale Romania 32, 00197 Roma.                      &          & Université Paris 7,  \\
        & Email: acretarola@luiss.it,                         &  & Email: pham@math.jussieu.fr\\
        & fgozzi@luiss.it                              &          & peter.tankov@polytechnique.org \\
        &                                               &     $^{3)}$     &  CREST-ENSAE,  \\
        &                                              &                        & and Institut Universitaire de France  \\[1ex]
\end{tabular} }

\begin{abstract}
We investigate optimal consumption policies in the liquidity risk
model introduced in~\cite{pt1}.
Our main result is to derive smoothness $C^1$  results for the
value functions of the portfolio/consumption choice problem. As an
important consequence,  we can prove the existence of the optimal
control (portfolio/consumption strategy) which we characterize
both in feedback form  in terms of the derivatives of the value
functions and as the solution of a second-order ODE. Finally,
numerical illustrations of the behavior of optimal consumption
strategies between two trading dates are given.
\end{abstract}

\vspace{5mm}

\noindent {\bf Key words~:} Illiquid market, optimal consumption,
integrodifferential equations, viscosity solutions, semiconcavity,
sub(super) differentials, optimal control.

\vspace{5mm}

\noindent {\bf  JEL Classification~:} G11

\vspace{2mm}

\noindent {\bf MSC Classification (2000)~:} 49K22, 49L25,  35F20,
91B28.

\newpage

\section{Introduction}

We investigate the optimal consumption policies in the
portfolio/consumption choice problem introduced in~\cite{pt1}.  In
this model, the investor has access to a market in which an
illiquid asset (stock or fund) is traded. The price of the asset
can be observed and trade orders can be passed only at random
times given by an exogenous Poisson process. These times model the
arrival of buy/sell orders in an illiquid market, or the dates on
which the results of a hedge fund are published. More generally,
these times may correspond to the dates on which the performance
of certain investment projects becomes known. The investor is also
allowed to consume (or distribute dividends to shareholders)
continuously from the bank account and the objective is to
maximize the expected discounted utility from consumption. The
resulting optimization problem is a nonstandard mixed
discrete/continuous time stochastic control problem, which leads
via the dynamic programming principle to a coupled system of
nonlinear integro-partial differential equations (IPDE).

In~\cite{pt}, the authors proved  that the value functions to this
stochastic control problem are characterized as the unique
viscosity solutions to the corresponding coupled IPDE. This
characterization makes the computation of value functions possible
(see~\cite{pt1}), but it does not yield the optimal consumption
policies in explicit form. In this paper, we go beyond the
viscosity property, and focus on the regularity of the value
functions. Using arguments of (semi)concavity and the strict
convexity of the Hamiltonian for the IPDE in connection with
viscosity solutions, we show that the value functions are
continuously differentiable. This regularity result is obtained
partly by adapting a technique introduced in~\cite{cs1} (see
also~\cite[p. 80]{bd}) and partly by a kind of bootstrap argument
that exploits carefully the special structure of the problem. This
allows then to get the existence of an optimal control  through a
verification theorem and to produce two characterizations of the
optimal consumption strategy: in feedback form in terms of the
classical derivatives of the value functions, and as the solution
of the Euler-Lagrange ordinary differential equation. We then use
these characterizations to study the properties of the optimal
consumption policies and to produce numerical examples, both in
the stationary and in the nonstationary case.

Portfolio optimization problems with discrete trading dates were
studied by several authors, but the profile of optimal consumption
strategies between the trading interventions has received little
attention so far. Matsumoto~\cite{mat06} supposes that the trades
succeed at the arrival times of an exogenous Poisson process but
does not allow for consumption. Rogers~\cite{rogers01} considers
an investor who can trade at discrete times and assumes that the
consumption rate is constant between the trading dates. Finally,
Rogers and Zane~\cite{rogzan02} allow the investor to change the
consumption rate between the trading dates and derive the HJB
equation for the value function but do not compute the optimal
consumption policy.

The rest of the paper is structured as follows. In section
\ref{model}, we rephrase the main assumptions of the liquidity
risk model introduced in~\cite{pt1}, introduce the necessary
definitions, and recall the viscosity characterization of the
value function. Section \ref{vproperties} establishes some new
properties of the value function such as the scaling relation.
Section \ref{sectionreg} contains the main result of the paper,
proving the regularity of the value function, which is used in
section \ref{verth} to characterize and study the optimal
consumption policies.  Some numerical illustrations depict the
behavior of  the consumption policies between two trading dates.
The technical proofs of some lemmas and propositions can be found
in the appendix.

\section{Formulation of the problem}\label{model}

\noindent Let us fix a probability space $(\Omega,\F,\P)$ endowed
with a filtration $\mathbb F=(\F_t)_{t\geq 0}$ satisfying the
usual conditions. All stochastic processes involved in this paper
are defined on the stochastic basis $(\Omega,\F,\mathbb
F,\P)$.\\
We consider a model of an illiquid market where the investor can
observe the positive stock price process $S$ and trade only at
random times $\{\tau_k\}_{k\geq 0}$ with
$\tau_0=0<\tau_1<\ldots<\tau_k<\ldots$. For simplicity, we assume
that $S_0$ is known and we denote by
$$
Z_k=\frac{S_{\tau_k}-S_{\tau_{k-1}}}{S_{\tau_{k-1}}},\quad k \geq
1,
$$
the observed return process valued in $(-1,+\infty)$, where we set
by convention $Z_0$ equal to some fixed constant.\\
The investor may also consume continuously from the bank account
(the interest rate $r$ is assumed w.l.o.g. to be zero) between two
trading dates. We introduce the continuous observation filtration
$\mathbb G^c=(\G_t)_{t \geq 0}$ where:
$$
\G_t=\sigma\{(\tau_k,Z_k):\tau_k \leq t)\},
$$
and the discrete observation filtration $\mathbb
G^d=(\G_{\tau_k})_{k \geq
0}$. Notice that $\G_t$ s trivial for $t<\tau_1$.\\
A control policy is a mixed discrete-continuous process
$(\alpha,c)$, where $\alpha=(\alpha_k)_{k \geq 1}$ is real-valued
$\mathbb G^d$-predictable, i.e. $\alpha_k$ is
$\G_{\tau_{k-1}}$-measurable, and $c=(c_t)_{t \geq 0}$ is a
nonnegative $\mathbb G^c$-predictable process: $\alpha_k$
represents the amount of stock invested for the period
$(\tau_{k-1},\tau_k]$ after observing the stock price at time
$\tau_{k-1}$, and $c_t$ is the consumption rate at time $t$ based
on the available information. Starting from an initial capital $x
\geq 0$, and given a control policy $(\alpha,c)$, we denote by
$X_k^x$ the wealth of investor at time $\tau_k$ defined by:
\begin{equation}\label{wealth}
X_k^x=x-\int_0^{\tau_k}c_t\ud t+\sum_{i=1}^k\alpha_iZ_i, \quad k
\geq 1, \quad X_0^x=x.
\end{equation}
\begin{definition}
Given an initial capital $x \geq 0$, we say that a control policy
$(\alpha,c)$ is {\em admissible}, and we denote $(\alpha,c) \in
\A(x)$ if
\begin{equation*}
X_k^x \geq 0, \quad {\rm a.s.}\quad \forall k \geq 1.
\end{equation*}
\end{definition}
\noindent According to~\cite{pt1,pt}, we assume the following
conditions on $(\tau_k,Z_k)$ stand in force from now on.
\begin{ass}\label{H}
\begin{itemize}
\item[]
\item[a)] $\{\tau_k\}_{k\geq 1}$ is the sequence of jumps
of a Poisson process with intensity $\lambda$.

\item[b)] (i) For all $k \geq 1$, conditionally on the
interarrival time $\tau_k-\tau_{k-1}=t\in \R_+$, $Z_k$ is
independent from $\{\tau_i,Z_i\}_{i<k}$ and has a distribution
denoted by $p(t,\ud z)$.\\
(ii) For all $t\geq 0$, the support of $p(t,\ud z)$ is
\begin{itemize}
\item[-] either an interval with interior equal to $(-\underline z,\bar
z)$, $\underline z \in (0,1]$ and $\bar z \in (0,+\infty]$;
\item[-] or it is finite equal to $\{-\underline z,\ldots,\bar
z\}$, $\underline z \in (0,1]$ and $\bar z \in (0,+\infty)$.
\end{itemize}
\item[c)] $\int zp(t,\ud z) \geq 0$, for all $t\geq 0$,
and there exist some $k \in \R_+$ and $b \in \R_+$, such that
$$
\int(1+z)p(t,\ud z) \leq ke^{bt},\quad \forall t \geq 0.
$$
\item[d)] The following continuity condition is fulfilled by the measure
$p(t,\ud z)$:
$$
\lim_{t\to t_0}\int w(z)p(t,\ud z)=\int w(z)p(t_0,\ud z), \quad
\forall t_0 \ge 0,
$$
for all measurable functions $w \in (-\underline z,\bar z)$ with
linear growth condition.
\end{itemize}
\end{ass}
\noindent The following simple but important examples illustrate
Assumption \ref{H}.

\begin{example}\label{black}
$S$ is extracted from a Black-Scholes model: $\ud S_t=bS_t\ud
t+\sigma S_t \ud W_t$, with $b \geq 0$, $\sigma >0$. Then $p(t,\ud
z)$ is the distribution of
$$
Z(t)=\exp \left[\left(b-\frac{\sigma^2}{2}\right)t+\sigma
W_t\right]-1,
$$
with support $(-1,+\infty)$ and condition c) of Assumption \ref{H}
is clearly satisfied, since in this case $\int(1+z)p(t,\ud
z)=\esp{\exp\left((b-\sigma^2/2)t+\sigma W_t\right)}=e^{bt}$.
\end{example}

\begin{example} \label{exsta}
$Z_k$ is independent of the waiting times $\tau_k-\tau_{k-1}$, in
which case its distribution $p(\ud z)$ does not depend on $t$. In
particular $p(\ud z)$ may be a discrete distribution with support
$\{z_0,\ldots,z_d\}$ such that $\underline z=-z_0 \in (0,1]$ and
$z_d=\bar z \in (0,+\infty)$.
\end{example}

\noindent We are interested in the optimal portfolio/consumption
problem:
\begin{equation} \label{vdef}
v(x)= \sup_{(\alpha,c)\in \A(x)}\esp{\int_0^{+\infty}e^{-\rho
t}U(c_t)\ud t},\quad x \geq 0,
\end{equation}
where $\rho$ is a positive discount factor and $U$ is an utility
function defined on $\R_+$. We introduce the following assumption:
\begin{ass}\label{h2}
The function $U$ is strictly increasing, strictly concave and
$C^1$ on $(0,+\infty)$ satisfying $U(0)=0$ and the Inada
conditions $U^\prime(0^+)=+\infty$ and $U^\prime(+\infty)=0$.
Moreover, $U$ satisfies the following growth condition: there
exists $\gamma \in (0,1)$ s.t.
\begin{equation} \label{growthcondition}
U(x) \leq K_1x^\gamma, \quad x \geq 0,
\end{equation}
for some positive constant $K_1$. In addition, condition (4.1)
of~\cite{pt} is satisfied, i.e.
\begin{equation*}
\rho > b\gamma + \lambda \left(\frac{k^\gamma}{\underline
z^\gamma}-1\right),
\end{equation*}
where $\gamma \in (0,1)$ and $k,b \in \R_+$ are provided by
Assumption \ref{H}.
\end{ass}
\noindent We denote by $\tilde U$ the convex conjugate of $U$,
i.e.
\begin{equation*} 
\tilde U(y)=\sup_{x>0}[U(x)-xy], \quad y \geq 0.
\end{equation*}
We note that $\tilde U$ is strictly convex under our assumptions
(see Theorem 26.6, Part V in~\cite{r}).
\begin{remark}
In~\cite{pt1,pt}, $U$ is supposed to be nondecreasing and concave
while here $U$ is strictly increasing and strictly convex. This
assumption is not very restrictive, since the most common utility
functions (like the ones of the CRRA type) satisfy it. \\
The main reason of this new hypothesis is that it implies the
strict concavity of the function $\tilde U$, which is a key
assumption to get the regularity of the value functions to our control problem. 
\end{remark}
\noindent Following~\cite{pt}, we consider the following version
of the dynamic programming principe (in short DPP) adapted to our
context
\begin{equation}\label{dpp}
v(x)=\sup_{(\alpha,c)\in \A(x)}\esp{\int_0^{\tau_1}e^{-\rho
t}U(c_t)\ud t+e^{-\rho \tau_1}v\left(X_1^x\right)}, \quad
\tau_1>0.
\end{equation}
This DPP is proved rigorously in Appendix of~\cite{pt}. From the
expression \eqref{wealth} of the wealth, and the measurability
conditions on the control, the above dynamic programming relation
is written as
\begin{equation}\label{vbis}
v(x)=\sup_{(a,c)\in \A_d(x)}\esp{\int_0^{\tau_1}e^{-\rho
t}U(c_t)\ud t+e^{-\rho \tau_1}v\left(x-\int_0^{\tau_1}c_t\ud t+a
Z_1\right)},
\end{equation}
where $\A_d(x)$ is the set of pairs $(a,c)$ with $a$ deterministic
constant, and $c$ a deterministic nonnegative process s.t. $a \in
[-x/\bar z,x/\underline z]$ and
\begin{equation}\label{la}
\int_0^tc_u\ud u \leq x-l(a) \quad {\rm i.e.}\quad
x-\int_0^tc_u\ud u+az \geq 0, \quad \forall t \geq 0,\ \forall z
\in (-\underline z, \bar z),
\end{equation}
where $l(a)=\max(a\underline z,-a\bar z)$ with the convention that
$\max (a\underline z,-a\bar z)=a\underline z$ when $\bar
z=+\infty$ (see Remark 2.3 of~\cite{pt1,pt} for further details).
Given $a \in [-x/\bar z,x/\underline z]$, we denote by $\C_a(x)$
the set of deterministic nonnegative processes satisfying
\eqref{la}. Moreover under conditions a) and b) of Assumption
\ref{H}, it is possible to write more explicitly the
right-hand-side of \eqref{vbis}, so that:
$$
v(x)=\sup_{\scriptsize \begin{array}{cc} a \in \big[-\frac{x}{\bar
z},\frac{x}{\underline z}\big]\\
c \in \C_a(x)\end{array}}\int_0^{+\infty}
e^{-(\rho+\lambda)t}\Big[U(c_t)+\lambda \int
v\Big(x-\int_0^tc_s\ud s+az\Big)p(t,\ud z)\Big]\ud t
$$
(see the details in Lemma 4.1 of~\cite{pt}). Let
$$
\D=\R_+ \times \X \quad \mbox{with}\quad \X=\left\{(x,a)\in
\R_+\times A\ :x \geq l(a)\right\},
$$
by setting $A=\R$ if $\bar z < +\infty$ and $A=\R_+$ if $\bar
z=+\infty$. Then, according to~\cite{pt1,pt}, we introduce the
dynamic auxiliary control problem: for $(t,x,a) \in \D$
\begin{equation}\label{vhat}
\hat v(t,x,a)=\sup_{c \in \C_a(t,x)}\int_t^{+\infty}
e^{-(\rho+\lambda)(s-t)}\left[U(c_s)+\lambda\int
v\left(Y_s^{t,x}+az\right)p(s,\ud z)\right]\ud s,
\end{equation}
where $\C_a(t,x)$ is the set of deterministic nonnegative
processes $c=(c_s)_{s\geq t}$, such that
$$
\int_t^sc_u\ud u\leq x-l(a),\quad {\rm i.e.}\quad Y_s^{t,x}+az\geq
0,\quad \forall s \geq t,\ \forall z\in (\underline z,\bar z),
$$
and $Y^{t,x}$ is the deterministic controlled process by $c \in
\C_a(t,x)$:
$$
Y_s^{t,x}=x-\int_t^sc_u\ud u,\quad s\geq t.
$$
In particular if we consider the function $g:\D \longrightarrow
\R_+$ defined by:
\begin{equation}\label{gfunction}
g\left(t,x,a\right):=\lambda\int v\left(x+az\right)p(t,\ud z),
\end{equation}
we can rewrite \eqref{vhat} as follows
\begin{equation}\label{vhat1}
\hat v(t,x,a)=\sup_{c \in \C_a(t,x)}\int_t^{+\infty}
e^{-(\rho+\lambda)(s-t)}\left[U(c_s)+g\left(s,Y_s^{t,x},a\right)\right]\ud
s.
\end{equation}
We know that the original value function is related to the
auxiliary optimization problem by:
\begin{equation} \label{valuev}
v(x)=\sup_{a \in \left[-x/\bar z,x/\underline z\right]}\hat
v(0,x,a).
\end{equation}
The Hamilton-Jacobi (in short HJ) equation associated to the
deterministic problem \eqref{vhat} is the following Integro
Partial Differential Equation (in short IPDE):
\begin{equation}\label{ipde1}
(\rho + \lambda)\hat v(t,x,a)-\frac{\partial \hat
v(t,x,a)}{\partial t}-\tilde U\left(\frac{\partial \hat
v(t,x,a)}{\partial x}\right)-\lambda\int v(x+az)p(t,\ud z)=0,
\end{equation}
with $(t,x,a)\in \D$. In terms of the function $g$:
\begin{equation}\label{ipde}
(\rho + \lambda)\hat v(t,x,a)-\frac{\partial \hat
v(t,x,a)}{\partial t}-\tilde U\left(\frac{\partial \hat
v(t,x,a)}{\partial x}\right)-g(t,x,a)=0,\quad (t,x,a)\in \D.
\end{equation}
\noindent In~\cite{pt}, the authors have already proved some basic
properties of the value function $\hat v$ as finiteness,
concavity, monotonicity and continuity on $\D$ (see Corollary 4.1
and Proposition 4.2).
In particular the authors have characterized the value function
through its dynamic programming equation by means of viscosity
solutions (see Theorem 5.1). \\
Our aim is to prove the smoothness of the value function $\hat v$
in order to get a verification theorem that provides the existence
(and uniqueness) of the optimal control feedback. We first prove
some further properties of the value functions $(v,\hat v)$ (as
strict monotonicity, uniform continuity on $\D$: see Section
\ref{vproperties}. Then we will study the regularity in the
stationary case, i.e. when $\hat v$ does not depend on $t$.
Finally we will extend the results to the general case. In
particular we will provide some regularity properties by means of
semiconcavity and
bilateral solutions.  \\
\indent It is helpful to recall the following definitions and
basic results from nonsmooth analysis concerning the generalized
differentials.
\begin{definition}
Let $u$ be a continuous function on an open set $D \subset
\Omega$. For any $y \in D$, the sets
\begin{align*}
D^-u(y)&=\bigg\{p\in \Omega:\liminf_{z \in D, z\to y}\frac{u(z)-u(y)-\langle p,z-y\rangle}{|z-y|}\geq 0\bigg\},\\
D^+u(y)&=\bigg\{p\in \Omega:\limsup_{z \in D, z\to
y}\frac{u(z)-u(y)-\langle p,z-y\rangle}{|z-y|}\leq 0\bigg\}
\end{align*}
are called respectively, the (Fréchet) {\em superdifferential} and
{\em subdifferential} of $u$ at $y$.
\end{definition}
\noindent The next lemma provides a description of $D^+u(x)$,
$D^-u(x)$ in terms of test functions.
\begin{lemma} \label{test}
Let $u \in C(D)$, $D \subset \Omega$ open set. Then,
\begin{enumerate}
\item $p \in D^+u(y)$ if and only if there exists $\varphi \in
C^1(D)$ such that $D\varphi(y)=p$ and $u-\varphi$ has a local
maximum at $y$;
\item $p \in D^-u(y)$ if and only if there exists $\varphi \in
C^1(D)$ such that $D\varphi(y)=p$ and $u-\varphi$ has a local
minimum at $y$.
\end{enumerate}
\end{lemma}
\begin{proof}
See Lemma II.1.7 of~\cite{bd} for the proof.
\end{proof}
\noindent As a direct consequence of Lemma \ref{test}, we can
rewrite Definition 5.1 of~\cite{pt} of viscosity solution adapted
to our context, in terms of sub and superdifferentials.

\begin{definition}\label{viscosity}
The pair of value functions $(v,\hat v) \in C_+(\R_+) \times
C_+(\D)$ given in \eqref{vdef}-\eqref{vhat} is a viscosity
solution to \eqref{valuev}-\eqref{ipde} if:
\begin{itemize}
\item[(i)] {\em viscosity supersolution property}: $v(x) \geq \sup_{a \in [-x/\bar z,x/\underline z]}
\hat v(0,x,a)$ and for all $a \in \A$,
\begin{equation} \label{super}
(\rho + \lambda)\hat v(t,x,a)-q-\tilde U(p)-g(t,x,a) \leq 0,
\end{equation}
for all $(q,p) \in \D_{t,x}^-\hat v(t,x,a)$, for all $(t,x,a) \in
\D$.
\item[(ii)] {\em viscosity subsolution property}: $v(x) \leq \sup_{a \in [-x/\bar z,x/\underline z]}
\hat v(0,x,a)$ and for all $a \in \A$,
\begin{equation}\label{sub}
(\rho + \lambda)\hat v(t,x,a)-q-\tilde U(p)-g(t,x,a) \geq 0,
\end{equation}
for all $(q,p) \in \D_{t,x}^+\hat v(t,x,a)$, for all $(t,x,a) \in
\D$.
\end{itemize}
The pair of functions $(v,\hat v)$ will be called a {\em viscosity
solution} of \eqref{valuev}-\eqref{ipde} if \eqref{super} and
\eqref{sub} hold simultaneously.
\end{definition}
\noindent Hence, we can reformulate the viscosity result stated
in~\cite{pt}.
\begin{proposition}
Suppose Assumptions \ref{H} and \ref {h2} stand in force. The pair
of value functions $(v,\hat v)$ defined in
\eqref{vdef}-\eqref{vhat} is the unique viscosity solution to
\eqref{valuev}-\eqref{ipde} in the sense of Definition
\ref{viscosity}.
\end{proposition}
\begin{proof}
See Theorem 5.1 of~\cite{pt} for a similar proof.
\end{proof}
\section{Some properties of the value functions}\label{vproperties}

In this section we discuss and prove some basic properties (strict
monotonicity, uniform continuity on $\D$) of the value functions
$(v,\hat v)$. We will always suppose Assumptions \ref{H} and
\ref{h2} throughout this section.\\
By Proposition 4.2 of~\cite{pt}, we already know that $v$ is
nondecreasing, concave and continuous on $\R_+$, with $v(0)=0$.
Moreover by Corollary 4.1 of~\cite{pt}, $v$ satisfies a growth
condition, i.e. there exists a positive constant $K$ such that
\begin{equation} \label{vgrowth}
v(x) \leq Kx^\gamma,\quad \forall x \geq 0.
\end{equation}

\noindent Here we provide the following properties on the function
$v$ and $g$ respectively whose proof can be found in Appendix:
\begin{proposition} \label{vstrincr}
The value function $v$ is strictly increasing on $\R_+$.
\end{proposition}

\noindent Now recall the function $g$ given in \eqref{gfunction}.
\begin{lemma} \label{gconcave}
The function $g$ is:
\begin{itemize}
\item[(i)] continuous in $t \in \R_+$, for every $(x,a) \in \X$;
\item[(ii)] strictly increasing in $x \in[l(a),+\infty)$, for every $a \in \A$ and $t \in \R_+$;
\item [(iii)] concave in $(x,a) \in \X$.
\end{itemize}
If we do not assume condition d) of Assumption \ref{H}, then the
function $g$ is only measurable in $t$ while (ii) and (iii) still
hold.
\end{lemma}

\noindent To conclude this section, we discuss a property of the
value function $\hat v$. We already know by Proposition 4.2
of~\cite{pt}, that $\hat v$ is concave and continuous in $(x,a)\in
\X$, and that has the following representation on the boundary
$\partial \X$:
\begin{equation}\label{generalboundary}
\hat v(t,x,a)=\int_t^{+\infty}e^{-(\rho+\lambda)(s-t)}g(s,x,a)\ud
s, \quad \forall t\geq 0,\quad \forall (x,a) \in \partial\X.
\end{equation}
In addition, by Corollary 4.1 of~\cite{pt}, we know that there
exists a constant $K$ that provides the following growth estimate:
\begin{equation}\label{hatvgrowth}
\hat v(t,x,a) \leq K\left(e^{bt}x\right)^\gamma, \quad \forall
(t,x,a) \in \D,
\end{equation}
with $\gamma \in (0,1)$ and $b$ is the constant given in condition
c) of Assumption \ref{H}.

\begin{lemma}
The value function $\hat v$ is strictly increasing in $x$, for
every $x \geq l(a)$, given $a \in A$.
\end{lemma}

\begin{proof}
The proof follows from the same arguments of the proof of
Proposition \ref{vstrincr} (see Appendix), using the strict
monotonicity of $U$ in $c$ and of $g$ in $x$ respectively.
\end{proof}

\subsection{The scaling relation for power utility}
In the case where the utility function is given by
$$
U(x) = K_1 x^\gamma, \;\;\; 0 < \gamma < 1, 
$$
using the fact that $c \in \mathcal C_a(t,x)$ if and only if
$\beta c \in \mathcal C_{\beta a}(t,\beta x)$ for any
$\beta>0$, we can easily deduce from the decoupled dynamic
programming principle in ~\cite{pt1} a scaling relation for the
value function $v$ and the auxiliary value function $\hat v$:
$$
\hat v(t,\beta x, \beta a) = \beta^\gamma \hat v(t,x,a),\qquad
v(\beta x) = \beta^\gamma v(x).
$$
This shows that the value function has the same form as in the
Merton model (confirmed by the graphs in \cite{pt1})  and that the
optimal investment strategy consists in investing a fixed
proportion of the wealth into the risky asset. In the case $\bar z$ $=$ $\infty$, $a$ is nonnegative and we can therefore reduce
the dimension of the problem and denote 
$$
v(x) = \vartheta_1 x^\gamma,\quad \hat v(t,x,a) = a^\gamma \bar
v(t,\xi),\quad \xi = x/a
$$
The equation satisfied by the auxiliary value function then
becomes
\begin{align*}
&(\rho+\lambda)\bar v - \frac{\partial \bar v}{\partial t} - \tilde
U\left(\frac{\partial \bar v}{\partial \xi}\right) - \lambda
\vartheta_1 \int (\xi+z)^\gamma p(t,dz)=0,
\\
&\vartheta_1 = \sup_{\xi\geq \underline{z}} \xi^{-\gamma}\bar v(0,\xi),
\end{align*}
in the nonstationary case and
\begin{align*}
&(\rho+\lambda)\bar v - \tilde U\left(\frac{\partial \bar
v}{\partial \xi}\right) - \lambda \vartheta_1 \int (\xi+z)^\gamma
p(dz)=0,\\
&\vartheta_1 = \sup_{\xi\geq \underline{z}} \xi^{-\gamma}\bar v(\xi),
\end{align*}
in the stationary case, with
$$
\tilde U(y) = \tilde K_1 y^{-\tilde \gamma},\quad \tilde \gamma =
\frac{\gamma}{1-\gamma}.
$$

\section{Regularity of the value functions} \label{sectionreg}

In this section we investigate the regularity property of the
value functions $(v, \hat v)$ in order to provide a feedback
representation form for the optimal strategies. Throughout the
whole section we will let Assumptions \ref{H} and \ref{h2} stand
in force.

\subsection{The stationary case} \label{regstat}

\noindent We start the study of the regularity with the simple
case when the distribution $p(t,\ud z)$ of the observed return
process $Z_k$, $k \geq 1$, does not depend on $t$, i.e. $p(t,\ud
z)=p(\ud z)$, for every $t \geq 0$, as in Example \ref{exsta}.
Then $g$ and $\hat v$ are independent of $t$ and the IPDE
\eqref{ipde} reduces to the integro ordinary differential equation
(in short IODE) for $\hat v(x,a)$:
\begin{equation}\label{iode}
(\rho + \lambda)\hat v(x,a)-\tilde U\left(\frac{\partial \hat
v(x,a)}{\partial x}\right)-g(x,a)=0,\quad (x,a)\in \X,
\end{equation}
where
\begin{align}
\hat v(x,a)&=\sup_{c \in \C_a(x)}\int_0^{+\infty}
e^{-(\rho+\lambda)s}\left[U(c_s)+\lambda\int
v\left(Y_s^x+az\right)p(\ud z)\right]\ud s \nonumber \\
&=\sup_{c \in \C_a(x)}\int_0^{+\infty}
e^{-(\rho+\lambda)s}\left[U(c_s)+g(Y_s^x,a)\right]\ud s
\label{vhatt}
\end{align}
with
\begin{equation}\label{vv}
v(x)=\sup_{a \in \left[-x/\bar z,x/\underline z\right]}\hat v(x,a)
\end{equation}
All the properties of the value function $\hat v$ discussed in the
previous section still hold for its restriction on the set $\X$.
In particular we have that $\hat v$ given in \eqref{vhatt} is
concave and continuous on $\X$, strictly increasing in $x \in
[l(a),+\infty)$ and satisfies the growth condition
\begin{equation*}
\hat v(x,a) \leq Kx^\gamma,\quad \forall (x,a) \in \X,
\end{equation*}
for some positive constant $K$, with $\gamma \in (0,1)$ and in
particular the condition on the boundary $\partial \X$ becomes:
\begin{equation*}
\hat v(x,a)= \int_0^{+\infty} e^{-(\rho+\lambda)s}g(x,a)\ud
s=\frac{1}{\rho+\lambda}g(x,a), \quad \forall (x,a) \in \partial
X.
\end{equation*}
We start by proving a first smoothness result for the function
$\hat v$.

\begin{proposition} \label{prosta}
The value function $\hat v$ defined in \eqref{vhatt} is $C^1$ with
respect to $x \in (l(a),+\infty)$, given $a \in A$. Moreover $\ds
\frac{\partial \hat v}{\partial x}(l(a)^+,a)=+\infty$.
\end{proposition}
\begin{proof}
We fix $a \in A$ and let us show that $\hat v$ is differentiable
on $(l(a),+\infty)$. First we note that the superdifferential
$D_x^+\hat v(x,a)$ is nonempty since $\hat v$ is concave. In view
of Proposition II.4.7 (c) of~\cite{bd}, since $\hat v$ is concave
in $x\in [l(a),+\infty)$, we just have to prove that for a given
$a\in A$, $D_x^+\hat v(x,a)$ is a
singleton for any $x \in (l(a),+\infty)$.\\
Suppose by contradiction that $p_1\neq p_2\in D_x^+\hat v(x,a)$.
Without loss of generality (since $x > l(a)$), we can assume that
$D_x^+\hat v(x,a)=[p_1,p_2]$. Denote by ${\rm co}D_x^*\hat v(x,a)$
the convex hull of the set
$$
D_x^*\hat v(x,a)=\left\{p:p=\lim_{n \to +\infty}D_x\hat v(x_n,a),\
x_n \to x\right\}.
$$
Since by Proposition II.4.7 (a) of~\cite{bd}, $D_x^+\hat
v(x,a)={\rm co}D_x^*\hat v(x,a)$, there exist sequences ${x_n}$,
${y_m}$ in $\R_+$ where $\hat v$ is differentiable and such that
$$
x=\lim_{n\to+\infty}x_n=\lim_{m\to+\infty}y_m,\
p_1=\lim_{n\to+\infty}D_x\hat v(x_n,a),\
p_2=\lim_{m\to+\infty}D_x\hat v(y_m,a).
$$
Since condition d) of Assumption \ref{H} and  Assumption \ref{h2}
hold, by Theorem 5.1 of~\cite{pt}, the pair of value functions
$(v,\hat v)$ is a viscosity solution to \eqref{iode}-\eqref{vv};
then by Proposition 1.9 (a) of~\cite{bd},
\begin{align*}
(\rho+\lambda)\hat v(x_n,a)&-\tilde U\left(D_x \hat
v(x_n,a)\right)-g(x_n,a)=0\\
(\rho+\lambda)\hat v(y_m,a)&-\tilde U\left(D_x \hat
v(y_m,a)\right)-g(y_m,a)=0;
\end{align*}
by continuity this yields
\begin{align}
(\rho+\lambda)\hat v(x,a)&-\tilde U\left(p_1\right)-g(x,a)=0 \label{p1}\\
(\rho+\lambda)\hat v(x,a)&-\tilde
U\left(p_2\right)-g(x,a)=0\label{p2}.
\end{align}
Now let $\bar p =\eta p_1+(1-\eta)p_2$, for $\eta \in (0,1)$.
Since $\bar p \in (p_1,p_2)$ $\subset$ $D_x^+\hat v(x,a) $, we
have by the viscosity supersolution property of $\hat v$~:
$$
(\rho+\lambda)\hat v(x,a)-\tilde U(\bar p)-g(x,a)\leq 0,
$$
so by \eqref{p1}-\eqref{p2}, we get
\begin{equation}\label{uconv}
\tilde U(\bar p)\geq \eta \tilde U(p_1)+(1-\eta)\tilde U(p_2).
\end{equation}
On the other hand, by strict convexity of $\tilde U$, we get
$$
\tilde U(\bar p)=\tilde U(\eta p_1+(1-\eta)p_2)< \eta \tilde
U(p_1)+(1-\eta)\tilde U(p_2),
$$
contradicting \eqref{uconv}. Hence $\hat v$ is differentiable at
any $x \in (l(a),+\infty)$. Notice from \eqref{iode} that for all
$a \in A$, $\displaystyle \frac{\partial \hat v}{\partial x}$ is
continuous in $x$\footnote{This follows also from Proposition
3.3.4 (e), pages 55-56 of~\cite{cs}.}. Now we prove the last
statement. If we get $x=l(a)$ in \eqref{vhatt}, then
$$
\hat v(l(a),a)=\frac{1}{\rho + \lambda}g(l(a),a).
$$
Now we send $x \to l(a)$ in \eqref{iode} (this is possible since
$\hat v$ and $g$ are continuous in $x \in [l(a),+\infty)$ and
since $\displaystyle \frac{\partial \hat v}{\partial x}$ is
monotone in $x$) and we obtain
$$
(\rho + \lambda)\hat v\left(l(a)^+,a\right)-\tilde
U\left(\frac{\partial \hat v \left(l(a)^+,a\right)}{\partial
x}\right)-g\left(l(a)^+,a\right)=0.
$$
Comparing the last formulas, we obtain
\begin{equation} \label{vderivinfty}
\tilde U\left(\frac{\partial \hat v
\left(l(a)^+,a\right)}{\partial x}\right)=0 \ \Longleftrightarrow
\ \frac{\partial \hat v \left(l(a)^+,a\right)}{\partial
x}=+\infty.
\end{equation}
\end{proof}

\noindent Before the final result we provide the following lemma.

\begin{lemma}
Let $v$ and $\hat v$ be the value functions given in \eqref{vdef}
and \eqref{vhatt} respectively. Then, given any $x> 0$ and calling
$a_x$ a maximum point of the problem \eqref{vv}, we have
\begin{equation}\label{eq:inclusionD+}
D^+v(x) \subseteq D_x^+\hat v(x,a_x).
\end{equation}
\end{lemma}

\begin{proof}
Let $x>0$. Since $v$ is concave we have
$$
D^+v(x)=\left\{p:v(x+h)-v(x)\leq ph, \; \forall h \mbox{ s.t.
}x+h\ge 0 \right\},
$$
Since $v$ is concave we have $D^+v(x) \neq \emptyset$. Let $p \in
D^+v(x)$. We have to prove that
\begin{equation}\label{eq:D+inclusion1}
    \hat v(x+h, a_x)-\hat v(x,a_x) \leq ph,
\end{equation}
for every $h$ such that $x+h\ge l(a_x)$. We first observe that
\begin{equation}\label{eq:D+inclusion2}
 \hat v(x+h, a_{x+h})-\hat v(x,a_x)=v(x+h)-v(x) \leq ph,
\end{equation}
for every $h$ such that $x+h\ge 0$ (here $a_x$ and $a_{x+h}$ are
optimal for $v(x)$ and $v(x+h)$ respectively).

\noindent Now call $I(x)= \left[-\frac{x}{\bar
z},\frac{x}{\underline{z}}\right]$ and observe that, for $0<x_1
<x_2$ we have ${0}\subset I(x_1)\subset I(x_2)$. So if $h\ge 0$ we
have that $a_x \in I(x+h)$, $\hat v(x+h, a_x)$ is well defined and
\begin{equation}\label{eq:D+inclusion3}
\hat v(x+h, a_x)\le \hat v(x+h,a_{x+h})
\end{equation}
which, together with \eqref{eq:D+inclusion2}, implies
\eqref{eq:D+inclusion1} for $h\ge 0$. Now if $x=l(a_x)$ there is
nothing more to prove. If $x>l(a_x)$ take $h<0$ such that $x+h\ge
l(a_x)$. For such $h$ we have $a_x \in I(x+h)$ so we still have
\eqref{eq:D+inclusion3} and so the claim as for the case $h>0$.
Hence $p \in D_x^+\hat v(x,a_x)$.
\end{proof}

\noindent Now we are ready to prove the final regularity result
for the stationary case.

\begin{theorem}
Let $v$, $\hat v$ be the value functions given in \eqref{vdef} and
\eqref{vhat} respectively. Then:
\begin{itemize}
\item $v\in
C^1(0,+\infty)$ and any maximum point in \eqref{vv} is internal
for every $x>0$; moreover $v'(0^+)=+\infty$;
\item
for every $a \in A$ we have $\hat v(\cdot,a) \in C^2
(l(a),+\infty)$. Finally $\ds \frac{\partial \hat v}{\partial
x}(l(a)^+,a)=+\infty$.
\end{itemize}
\end{theorem}

\begin{proof}
Since $v$ is concave then $D^+v(x)$ is nonempty at every $x>0$.
This implies, by \eqref{eq:inclusionD+}, that also $D_x^+\hat
v(x,a_x)$ is nonempty for every $x>0$. Since, by
\eqref{vderivinfty}, $\ds \frac{\partial\hat v}{\partial x}
(l(a)^+,a)=+\infty$ (which implies $D^+_x \hat
v(l(a),a)=\emptyset$) we get that it must be $x>l(a_x)$ and so any
maximum point in \eqref{vv} is internal. Moreover since, given $a
\in A$ we have that $\hat v$ is $C^1$ in $x\in (l(a),+\infty)$
then the superdifferential is a single point and so from
\eqref{eq:inclusionD+} also $D^+v(x)$ ia single point, which
implies the wanted regularity of $v$. The statement
$v'(0^+)=+\infty$ follows simply observing that $v(x) \ge \hat
v(x,0)$, $v(0)=\hat v(0,0)=0$, and from (\ref{vderivinfty}) for
$a$ $=$ $0$. Finally $\hat v(\cdot,a) \in C^2 (l(a),+\infty)$
follows from \eqref{iode} and $\ds \frac{\partial \hat v}{\partial
x}(l(a)^+,a)=+\infty$ from Proposition \ref{prosta}.
\end{proof}

\subsection{The nonstationary case} \label{regevol}

In this subsection we study the regularity of the value function
$\hat v$ in the general case where the distribution $p(t,\ud z)$
may depend on time. With respect to the stationary case,  the
value function $\hat v$ is in general not concave in both
time-space variables, and we cannot apply directly arguments as in
Proposition \ref{prosta}.  Actually, we shall prove the regularity
of the value function $\hat v$ as well as in the stationary case,
by means of (locally) semiconcave functions.

\noindent First, we recall the concept of semiconcavity. Let $S$
be a subset of $\Omega$.
\begin{definition}
We say that a function $u:S\to \R$ is {\em semiconcave} if there
exists a nondecreasing upper semicontinuous function
$\omega:\R_+\to\R_+$ such that $\lim_{\rho \to 0^+}\omega(\rho)=0$
and
\begin{equation} \label{semiconcavity}
\eta u(x_1)+(1-\eta)u(x_2)-u(\eta x_1+(1-\eta)x_2)\leq
\eta(1-\eta)|x_1-x_2|\omega(|x_1-x_2|),
\end{equation}
for any pair $x_1,x_2$ such that the segment $[x_1,x_2]$ is
contained in $S$ and for $\eta \in [0,1]$. In particular we call
{\em locally semiconcave} a function which is semiconcave on every
compact subset of its domain of definition.
\end{definition}

\noindent Clearly, a concave function is also semiconcave. An
important example of semiconcave functions is given by the smooth
ones.
\begin{proposition} \label{cione}
Let $u \in C^1(A)$, with $A$ open. Then  both $u$ and $-u$ are
locally semiconcave in $A$ with modulus equal to the modulus of
continuity of $Du$.
\end{proposition}
\begin{proof}
See Proposition 2.1.2 of~\cite{cs} for the proof.
\end{proof}
\begin{remark}
We should stress that the superdifferential of a locally
semiconcave function is nonempty, since all the properties of
superdifferential hold even locally.
\end{remark}
\noindent We introduce an additional assumption on the measure
$p(t,\ud z)$:
\begin{ass}\label{h5}
for every $a \in \A-\{0\}$ , the map
$$
(t,x) \longmapsto \lambda \int w(x+az)p(t,\ud z)
$$
is locally semiconcave for $(t,x) \in (0,+\infty)\times
(l(a),+\infty)$, and for all measurable continuous functions $w$
on $\R$ with linear growth condition.
\end{ass}
\begin{remark}\label{rm:hpH5}
Since it is not trivial to check the validity of Assumption
\ref{h5}, we give some conditions the guarantee it. First of all,
we exclude the case $a=0$ from Assumption \ref{h5} since in this
case we have, for every $(t,x)\in \R_+\times [l(a),+\infty)$
$$
g(t,x)=\lambda v(x)
$$
so we are in the stationary case and we already know from the
previous section that $\hat v$ is $C^1$. Now, when $a\not=0$, we
set the new variable $y=x+az=h_x(z)$ and call $\mu(t,x;\ud y)$ the
measure $(h_x\circ p)(t, \ud z)$. The measure $\mu$ has the
following support:
\begin{enumerate}
\item  $(x-a\underline
z,+\infty)$, if $\bar z=+\infty$, and $a>0$;\label{infinitesup}
\item $(x-a\underline z,x+a \bar z)$,
if $\bar z<+\infty$ and $a > 0$ \label{finitesup}
\item $(x+a\bar z,x-a\underline z)$,
if $\bar z<+\infty$ and $a<0$;
\item $\{x-a\underline z,\ldots,x+a\bar z\}$,
if the support of $p$ is finite and $a > 0$ (in this case $\bar
z<+\infty)$;
\item $\{x+a\bar z,\ldots,x-a\underline z\}$,
if the support of $p$ is finite and $a<0$ (in this case $\bar
z<+\infty)$.
\end{enumerate}
Now Assumption \ref{h5} can be written as: the function $g_w$
given by
$$
(t,x) \longmapsto \lambda \int w(y)\mu(t,x;\ud y)
$$
is locally semiconcave for $(t,x) \in (0,+\infty) \times
(l(a),+\infty)$, and for all measurable continuous functions $w$
on $\R$ with linear growth condition.\\
In this form, it is easier to find conditions that guarantee the
validity of this assumption in terms of the regularity of $\mu$.
For example, if we assume the measure $p(t,\ud z)$ has a density
$f(t,z)$, the integral
$$
\int w(x+az)f(t,z)\ud z
$$
by the above change of variable is rewritten as:
$$
\frac{1}{a}\int w(y)f\left(t,\frac{y-x}{a}\right)\ud y.
$$
Now, by Proposition \ref{cione}, the local semiconcavity of $g_w$
in the interior $(0,+\infty) \times (l(a),+\infty)$ of its domain follows from its continuous differentiability. \\
Let us give a condition that guarantees that $g_w$ is $C^1$ in the
case \ref{infinitesup}. If the density $f$ is continuously
differentiable and suitable integrability conditions are
satisfied, then we have: for every $a
> 0$,
\begin{align*}
\frac{\partial g_w(t,x)}{\partial
t}&=\frac{1}{a}\int_{x-a\underline z}^{+\infty}w(y)\frac{\partial
f}{\partial
t}\left(t,\frac{y-x}{a}\right)\ud y, \\
\frac{\partial g_w(t,x)}{\partial
x}&=-\frac{1}{a^2}\int_{x-a\underline
z}^{+\infty}w(y)\frac{\partial f}{\partial
x}\left(t,\frac{y-x}{a}\right)\ud y-\frac{1}{a}w(x-a\bar
z)f(t,\underline z),
\end{align*}
for $(t,x) \in (0,+\infty) \times (l(a),+\infty)$. From the above
expressions, it is easy to check that we can derive the continuous
differentiability from the following assumptions:
\begin{itemize}
\item the density $f$ is continuous and
for each $a \in \A$, the generalized integral
$$
\int_{x-a\underline z}^{+\infty} (1+|y|)
f\left(t,\frac{y-x}{a}\right)\ud y
$$
converges for every $(t,x) \in (0,+\infty) \times (l(a),+\infty)$;
\item the partial derivatives $\ds \frac{\partial f}{\partial t}$, $\ds \frac{\partial f}{\partial
x}$ are continuous and satisfy respectively the following
integrability conditions: for each $a \in \A$,
$$
\int_{x-a\underline z}^{+\infty} (1+|y|) \frac{\partial
f}{\partial t}\left(t,\frac{y-x}{a}\right)\ud y
$$
converges uniformly with respect to $t$ $\in$ $\T$, for any
compact set  $\T$ of  $(0,+\infty)$, for every $x \in
(l(a),+\infty)$, and
$$
\int_{x-a\underline z}^{+\infty} (1+|y|) \frac{\partial
f}{\partial x}\left(t,\frac{y-x}{a}\right)\ud y
$$
converges uniformly with respect to $x$ $\in$ $K$, for any compact
set $K$ of  $(l(a),+\infty)$, for every $t \in (0,+\infty)$.
\end{itemize}
Let us check the above assumptions in the Black-Scholes model,
introduced in Example \ref{black}. We recall that the dynamics of
$S$ is given by $\ud S_t=bS_t\ud t+\sigma S_t\ud W_t$, with $b
\geq 0$, $\sigma
>0$, so that $p(t,\ud z)$ is the distribution of
$$
Z(t)=\exp\left[\left(b-\frac{\sigma^2}{2}\right)t+\sigma
W_t\right]-1,
$$
with support $(-1,+\infty)$. Then, since $S$ has a lognormal
distribution, the density $f_Z$ is given by:
$$
f_Z(t,z)=\frac{1}{\sigma\sqrt{2\pi
t}(z+1)}\exp\left[-\frac{\left(\ln(z+1)-\left(b-\frac{\sigma^2}{2}\right)t\right)^2}{2\sigma^2
t}\right].
$$
We compute the partial derivatives $\ds \frac{\partial
f_Z}{\partial t},\ \frac{\partial f_Z}{\partial z}$ and we get:
\begin{align*}
\frac{\partial f_Z(t,z)}{\partial t}&=\frac{1}{2\sigma\sqrt{2\pi
t} (z+1)}e^{-\frac{(\ln(z+1)-(b-\frac{\sigma^2}{2})t)^2}{2\sigma^2
t}}\left[-\frac{1}{t}+\frac{1}{\sigma^2t}\ln
^2(z+1)-\frac{b}{\sigma^2}+\frac{1}{2}\right],\\
\frac{\partial f_Z(t,z)}{\partial z}&=\frac{1}{\sigma \sqrt{2\pi
t}(z+1)^2}e^{-\frac{(\ln(z+1)-(b-\frac{\sigma^2}{2})t)^2}{2\sigma^2
t}}\left[-\frac{1}{\sigma^2 t}\ln
(z+1)+\frac{b}{\sigma^2t}-\frac{3}{2}\right].
\end{align*}
Hence it is not difficult to check that the assumptions described
above are satisfied.
\end{remark}

\noindent We start by proving a smoothness property for $\hat v$.

\begin{proposition} \label{smoothness}
Suppose that Assumption \ref{h5} is satisfied. Then the value
function $\hat v$ defined in \eqref{vhat} belongs to
$C^1\left([0,+\infty) \times (l(a),+\infty)\right)$, given $a \in
\A$. Moreover
\begin{equation}\label{xder}
\ds \frac{\partial \hat v(t,l(a)^+,a)}{\partial x}=+\infty,\ {\rm
for\ every}\ t \ge 0.
\end{equation}
\end{proposition}

\begin{proof}
We fix $a \in \A$ and let us show that $\hat v$ is differentiable
at any $(t,x) \in (0,+\infty) \times (l(a),+\infty)$. When $a=0$,
as we noted at the beginning of Remark \ref{rm:hpH5}, $\hat v$ is
independent of $t$ and $C^1$ in $x$ thanks to the results of
Section 5. Take then $a \ne 0$. First, we notice from Assumption
\ref{h5} that $g$ is (locally) semiconcave in  $(t,x)\in
(0,+\infty) \times (l(a),+\infty)$. Together with the  concavity
of $U$, this shows that  $\hat v$ is (locally) semiconcave in
$(t,x) \in (0,+\infty) \times (l(a),+\infty)$. Indeed, if we set
$r=s-t$ we can rewrite \eqref{vhat1} as follows:
\begin{align*}
\hat v(t,x,a)&=\sup_{c \in \C_a(0,x)}\int_0^{+\infty}
e^{-(\rho+\lambda)r}\left[U(c_r)+g\left(r+t,Y_r^{0,x},a\right)\right]\ud
r\\
&=\sup_{c \in \C_a(x)}\int_0^{+\infty}
e^{-(\rho+\lambda)r}\left[U(c_r)+g\left(r+t,Y_r^x,a\right)\right]\ud
r.
\end{align*}
For every $(t,x) \in \R_+\times (l(a),+\infty), \ c \in \C_a(x)$,
we put
$$
J(t,x,a;c):=\int_0^{+\infty}
e^{-(\rho+\lambda)r}\left[U(c_r)+g\left(r+t,Y_r^x,a\right)\right]\ud
r.
$$
Let $t_1,t_2 > 0$, with $t_1 < t_2$, $x_1,x_2 \in (l(a),+\infty)$,
with $x_1 < x_2$.
By setting $t_\eta=\eta t_1+(1-\eta)t_2$, $x_\eta=\eta
x_1+(1-\eta) x_2$, we have for all $(t,x) \in (0,+\infty)\times
(l(a),+\infty)$
\begin{align*}
\eta &
J(t_1,x_1,a;c_1)+(1-\eta)J(t_2,x_2,a;c_2)-J(t_\eta,x_\eta,a;c_\eta)\\
& =\int_0^{+\infty}e^{-(\rho+\lambda)r}\left[\eta
U(c_1(r))+(1-\eta)U(c_2(r))-U(c_\eta(r))\right]\ud r\\
& +\int_0^{+\infty}e^{-(\rho+\lambda)r}\left[\eta
g\left(r+t_1,Y_r^{x_1},a\right)+(1-\eta)g\left(r+t_2,Y_r^{x_2},a\right)-g\left(r+t_\eta,Y_r^{x_\eta},a\right)\right]\ud
r\\
& < \int_0^{+\infty}e^{-(\rho+\lambda)r}\left[\eta
g\left(r+t_1,Y_r^{x_1},a\right)+(1-\eta)g\left(r+t_2,Y_r^{x_2},a\right)-g\left(r+t_\eta,Y_r^{x_\eta},a\right)\right]\ud
r,
\end{align*}
by using the strict concavity of $U$. By  the semiconcavity of the
function $g$ and by taking the supremum of the functional $J$ over
the set $\C_a(x)$, we can derive the semiconcavity of $\hat v$ for
$(t,x) \in (0,+\infty) \times (l(a),+\infty)$.
Hence $D_{t,x}^+\hat v(t,x,a) \neq \emptyset$, so we have just to
prove that $D_{t,x}^+\hat v(t,x,a)$ is a singleton, for each
$(t,x) \in (0,+\infty) \times (l(a),+\infty)$. \noindent By using
the same arguments of Proposition \ref{prosta}, we get the Fr\'echet differentiability. \\
By Proposition 3.3.4 (e), pages 55-56 of~\cite{cs}, we get the
continuity of the couple $\ds \left(\frac{\partial \hat
v}{\partial t},\frac{\partial \hat v}{\partial x}\right)$ for
$(t,x) \in (0,+\infty) \times (l(a), +\infty)$, given $a \in \A$.
Then the value function $\hat v$ defined in \eqref{vhat} belongs
to $C^1((0,+\infty) \times (l(a),+\infty))$, given $a \in \A$.\\
To get that $\hat v (\cdot,\cdot,a) \in C^1([0,+\infty) \times
(l(a),+\infty))$ it is enough to extend the datum $g$ (and so the
value function $\hat v$) to small negative times and repeat the
above arguments.

Now we prove \eqref{xder} by using similar arguments to the ones
to check the final statement of Proposition \ref{prosta}. If we
get $x=l(a)$ in \eqref{vhat1}, then
$$
\hat v(t,l(a),a)=\int_t^{+\infty}e^{-(\rho +
\lambda)(s-t)}g(s,l(a),a)\ud s,\quad \forall t \geq 0.
$$
Now we send $x \to l(a)$ in \eqref{ipde} (this is possible since
$\hat v$, $g$ and $\displaystyle \frac{\partial \hat v}{\partial
t}$ are continuous in $x \in [l(a),+\infty)$\footnote{By Remark
4.4 of \cite{pt} we already know that $\hat v$ is differentiable
in $t$ on the boundary and in particular the continuity follows
from \eqref{vhat}.} and since $\displaystyle \frac{\partial \hat
v}{\partial x}$ is monotone in $x$) and we obtain
$$
(\rho + \lambda)\hat v\left(t,l(a)^+,a\right)-\frac{\partial \hat
v\left(t,l(a)^+,a\right)}{\partial t}-\tilde U\left(\frac{\partial
\hat v\left(t,l(a)^+,a\right)}{\partial
x}\right)-g\left(t,l(a)^+,a\right)=0.
$$
Comparing the last formulas, we obtain
\begin{equation*} 
\tilde U\left(\frac{\partial \hat
v\left(t,l(a)^+,a\right)}{\partial x}\right)=0 \
\Longleftrightarrow \ \frac{\partial \hat
v\left(t,l(a)^+,a\right)}{\partial x}=+\infty, \quad \forall t
\geq 0.
\end{equation*}
\end{proof}

\begin{lemma}
Suppose that Assumption \ref{h5} is satisfied. Let $v$ and $\hat
v$ be the value functions given in \eqref{vdef} and \eqref{vhat}
respectively. Then, given any $x> 0$ and calling $a_x$ a maximum
point of the problem \eqref{valuev}, we have
\begin{equation*}
D^+v(x) \subseteq D_x^+\hat v(0,x,a_x).
\end{equation*}
\end{lemma}
\begin{proof}
It works exactly as well as in the stationary case.
\end{proof}

\noindent We come now to the final regularity result for the
nonstationary case.

\begin{theorem}
Suppose that Assumption \ref{h5} is satisfied. Let $v$, $\hat v$
be the value functions given in \eqref{vdef} and \eqref{vhat}
respectively. Then:
\begin{itemize}
\item $v\in
C^1(0,+\infty)$ and any maximum point in \eqref{vv} is internal
for every $x>0$; moreover $v'(0^+)=+\infty$;
\item
for every $a \in A$ we have $\hat v(\cdot, \cdot,a) \in
C^1\left([0,+\infty) \times (l(a),+\infty)\right)$; finally
\begin{equation*}
\ds \frac{\partial \hat v(t,l(a)^+,a)}{\partial x}=+\infty,\ {\rm
for\ every}\ t \ge 0.
\end{equation*}
\end{itemize}
\end{theorem}

\begin{proof}
It follows as in the stationary case.
\end{proof}

\begin{remark}\label{gregularity}
We should stress that even if the semiconcavity assumption
\ref{h5} does not hold, the continuous differentiability in $x$ of
the function $g$ given in \eqref{gfunction} is still guaranteed in
the case of power utility and when the density $p(t,\ud z)$ is
supposed to be ``sufficiently regular'' in $x$.
\end{remark}

\section{Existence and characterization of optimal strategies}\label{verth}
\vspace{2mm} Let Assumptions \ref{H}, \ref{h2} and  \ref{h5} stand
in force throughout this section.
\subsection{Feedback representation form of the optimal
strategies}

The following result guarantees the existence and uniqueness of
the optimal control for the auxiliary problem \eqref{vhat}.
\begin{proposition} \label{genuniqueness}
Let $\hat v$ be the value function given in \eqref{vhat}. Fix $a
\in A$. We denote by $I=(U^\prime)^{-1}:(0,+\infty) \to
(0,+\infty)$ the inverse function of the derivative $U^\prime$ and
we consider the following nonnegative measurable function for each
$a \in A$:
\begin{equation} \label{genchat}
\hat c(t,x,a)=I\left(\frac{\partial \hat v(t,x,a)}{\partial
x}\right)=\arg \max_{c\geq 0}\left[U(c)-c\frac{\partial \hat
v(t,x,a)}{\partial x}\right].
\end{equation}
Let $(t,x) \in \R_+ \times [l(a), +\infty)$. There exists a unique
optimal couple $(\bar c_\cdot,\bar Y_\cdot)$ at $(t,x)$ for the
auxiliary problem introduced in \eqref{vhat} given by:
\begin{equation}\label{gencbar}
\bar c_s:=\hat c(s,\bar Y_s,a), \quad s\ge t,
\end{equation}
where $\bar Y_s$, $s \geq t$, is the unique solution of
\begin{equation}\label{cpgen}
\left\{ \begin{array}{ll}
Y'_s=-\hat c(s,Y_s,a),\qquad s \geq t\\
Y_t=x.
\end{array}
\right.
\end{equation}
Note that the triplet $(s,\bar Y_s, a) \in \D$, for $s \geq t$.
\end{proposition}
\begin{proof}
A rigorous proof can be found in Appendix.
\end{proof}
\noindent  Under suitable assumptions, we state the verification
theorem for the coupled IPDE \eqref{valuev}-\eqref{ipde}, which
provides the optimal control in feedback form.
\begin{theorem} \label{genexistence}
There exists an optimal control policy $(\alpha^*,c^*)$ given by
\begin{align}
\alpha_{k+1}^*&=\arg \max_{-\frac{X_k^x}{\bar z}\leq a\leq \frac{X_k^x}{\underline z}}\hat v(0,X_k^x,a),\quad k\geq 0 \label{genalphastar}\\
c_t^*&=\hat c\left(t-\tau_k, Y^{(k)}_t,\alpha_{k+1}^*\right),
\quad \tau_k<t\leq \tau_{k+1}, \label{gencstar}
\end{align}
where $X_k^x$ is the wealth investor at time $\tau_k$ given in
\eqref{wealth} and $Y^{(k)}_\cdot $ is the unique solution of
\begin{equation}
\left\{ \begin{array}{ll}
Y'_s=-\hat c(s,Y_s,\alpha_{k+1}^*),\qquad  \tau_k<s \leq \tau_{k+1}\\
Y_{\tau_k}=X_k^x.
\end{array}
\right.
\end{equation}
\end{theorem}

\begin{proof}
Thanks to Proposition \ref{genuniqueness}, we can prove the
existence of an optimal feedback control $(\alpha^*,c^*)$ for
$v(x)$.\\
Given $x \geq 0$, consider the control policy $(\alpha^*,c^*)$
defined by \eqref{genalphastar}-\eqref{gencstar}. By construction,
the associated wealth process satisfies for all $k\geq 0$,
\begin{align*}
X_{k+1}^x&=X_k^x-\int_{\tau_k}^{\tau_{k+1}}c_s^*\ud
s+\alpha_{k+1}^*Z_{k+1}\\
&=Y^{(k)}_{\tau_{k+1}}+\alpha_{k+1}^*Z_{k+1}\\
&\geq l(\alpha_{k+1}^*)+\alpha_{k+1}^*Z_{k+1} \geq 0,\ {\rm a.s.}
\end{align*}
since $-\underline z\leq Z_{k+1}\leq \bar z$ a.s. Hence,
$(\alpha^*,c^*) \in \mathcal A(x)$, i.e. $(\alpha^*,c^*)$ is
admissible.
By Proposition \ref{genuniqueness} and definition of
$\alpha_{k+1}^*$ and $v$, we have:
\begin{align*}
&v(X_k^x)\\
& =\hat
v(0,X_k^x,\alpha_{k+1}^*)\\
&=\int_{\tau_k}^{+\infty}e^{-(\rho+\lambda)(s-\tau_k)}\left[U(\hat
c_s(\tau_k, Y_s^{(k)},\alpha_{k+1}^*))+g\right](s-\tau_k,
Y_s^{(k)},\alpha_{k+1}^*)\ud s\\
&=\esp{\int_{\tau_k}^{\tau_{k+1}}e^{-\rho(s-\tau_k)}U(c_s^*)\ud
s+e^{-(\rho+\lambda)(\tau_{k+1}-\tau_k)}v(X_{k+1}^x)\bigg{|}\G_{\tau_k}},
\end{align*}
by Lemma 4.1 of~\cite{pt}. By iterating these relations for all
$k$, and using the law of conditional expectations, we obtain
$$
v(x)=\esp{\int_0^{\tau_n}e^{-\rho s}U(c_s^*)\ud s+e^{-\rho
\tau_n}v(X_n^x)},
$$
for all $n$. By sending $n$ to infinity, we get:
$$
v(x)=\esp{\int_0^{+\infty}e^{-\rho s}U(c_s^*)\ud s},
$$
which provides the required result. \end{proof}

\begin{remark}
In the stationary case the Assumption \ref{h5} is not needed to
prove the existence of feedback controls, as it is automatically
satisfied. Moreover we note that in the stationary case there is
not an explicit dependence on $t$ of the optimal control in
feedback form. Indeed, it is given by the couple $(\alpha^*,c^*)$,
where
\begin{align*}
\alpha_{k+1}^*&=\arg \max_{-\frac{X_k^x}{\bar z}\leq a\leq \frac{X_k^x}{\underline z}}\hat v(X_k^x,a),\quad k\geq 0 \\
c_t^*&=\hat c\left( Y_t^{(k)},\alpha_{k+1}^*\right), \quad
\tau_k<t\leq \tau_{k+1}, 
\end{align*}
and in particular $\hat c$ is the restriction on the set $\X$ of
the nonnegative measurable functions introduced in
\eqref{genchat}, i.e.
\begin{equation}\label{chat}
\hat c(x,a)=I\left(\frac{\partial \hat v(x,a)}{\partial
x}\right)={\rm arg} \max_{c \ge 0}\left[U(c)-c\frac{\partial \hat
v(x,a)}{\partial x}\right].
\end{equation}
\end{remark}

\begin{remark} \label{uniquen}
It is not trivial to state the uniqueness of the strategy
$(a^*,c^*)$, whose existence is proved in Theorem
\ref{genexistence}. We can only say that, if we prove that $a^*$
is unique, then also
$c^*$ will be unique thanks to 
Theorem \ref{genexistence}. The problem is strictly related to the
behavior of the functions $\hat v$ and $g$ that are ex ante not
strictly concave in $a$.
\end{remark}

\begin{remark} \label{rm:strictconcavityofv}
>From the feedback representation given in Proposition
\ref{genuniqueness} and in Theorem \ref{genexistence}, it follows
that the function $v$ is strictly concave and that the functions
$g$ and $\hat v$ are strictly concave in $x$. Indeed, given two
points $x_1,x_2 > l(a)$ and calling $c_1^*, c_2^*$ the
corresponding optimal consumption paths for the original problem,
we have, for $\eta \in (0,1)$,
\begin{equation}\label{scv}
\begin{split}
v(\eta x_1 &+ (1-\eta) x_2)- \eta v(x_1)-(1-\eta)v(x_2)\\
\qquad & \ge \esp{\int_0^{+\infty} e^{-\rho s}\left[U(\eta
c_{1s}^* + (1-\eta) c_{2s}^*) -U(\eta c_{1s}^*) -  (1-\eta)U(
c_{2s}^*)\right]\ud s}.
\end{split}
\end{equation}
Thanks to the feedback formulas, the two consumption rates $c_1^*,
c_2^*$ must be different in a set of positive measure ($\ud t
\times \ud \P$) so the right-hand-side of \eqref{scv} is strictly
positive and we get strict concavity of $v$. Then the strict
concavity of $g$ in $x$ follows directly from its definition
whereas the strict concavity of $\hat v$ in $x$ follows from the
IPDE \eqref{ipde}.
\end{remark}

\subsection{Consumption policy between two trading dates}

>From the regularity properties discussed in Subsection
\ref{sectionreg}, we can deduce more properties of the optimal
consumption policy. We discuss them separately for the stationary
and the nonstationary case.

\subsubsection{The stationary case}

\begin{proposition}
Let $a\in A$ and $(t,x) \in \R_+ \times [l(a),+\infty) $. Let
$(\bar c_\cdot, \bar Y_\cdot )$ be the optimal couple for the
auxiliary problem starting at $(t,x)$. If $x=l(a)$, then $\bar
c\equiv 0$, so  $\bar Y \equiv l(a)$. If $x>l(a)$ then $\bar c$ is
continuous, strictly positive and strictly decreasing while $\bar
Y$ is strictly decreasing and strictly convex. Moreover $\lim_{t
\to +\infty} \bar c_t=0$ and $\lim_{t \to +\infty} \bar Y_t=l(a)$.
\end{proposition}

\begin{proof}
The first statement follows immediately from the setting of the
auxiliary problem. We prove the second statement. Indeed, by
\eqref{chat} and Remark \ref{rm:strictconcavityofv} it follows
that the function $\hat c$ is strictly increasing and continuous
in $x$. Since $\bar c_t=\hat c(\bar Y_t,a)$ and $\bar Y$ is
continuous and decreasing, then also $\bar c$ is decreasing.
Moreover, $\bar c_t >0$ for every $t$: indeed if it becomes zero
in finite time then the associated costate would have a
singularity and this is impossible: see the proof of Proposition
\ref{pr:regdiff} in the non stationary case. The strict positivity
of $\bar c$ implies that $\bar Y$ is strictly decreasing and so,
by \eqref{chat} that $\bar c $ is strictly decreasing and $\bar Y$
is strictly convex.

\noindent Finally, by the definition of the auxiliary control
problem, $\int_0^{+\infty} \bar c_s\ud s \leq x-l(a)$ which
implies the limit of $\bar c$. If the limit of $\bar Y$ is $x_1
>l(a)$, we get from the feedback formula \eqref{gencbar} that
$$ \lim_{t \to +\infty} \bar c_t= \hat c(x_1,a)>0$$ which is impossible.
\end{proof}

\noindent The regularity results for $c$ then allow to deduce an
autonomous equation for the optimal consumption policy between two
trading dates.
\begin{proposition}
Suppose that $U\in C^2((0,\infty))$ with $U''(x)<0$ for all $x$.
Then the wealth process $Y$  between two trading dates is twice
differentiable and satisfies the second-order ODE
\begin{equation}\label{auto}
\frac{\ud^2 Y_t}{\ud t^2} = \frac{g'(Y_t) -
(\rho+\lambda)U'(c_t)}{U''(c_t)},\quad c_t = -\frac{\ud Y_t}{\ud
t}.
\end{equation}
\end{proposition}
\begin{proof}
Differentiating equations \eqref{iode} and \eqref{chat} with
respect to $x$ and \eqref{cpgen} (restricted on $\X$) with respect
to $t$, we obtain
\begin{align*}
&\frac{\ud^2 Y_t}{\ud t^2 } = \frac{\partial \hat c(Y_t,a)}{\partial x} c_t,\\
&\frac{\partial \hat c(x,a)}{\partial x} = I'\left(\frac{\partial
\hat v(x,a)}{\partial x}\right)\frac{\partial^2 \hat
v(x,a)}{\partial x^2} = \frac{1}{U''(\hat
c(x,a))}\frac{\partial^2\hat v(x,a)}{\partial x^2},
\\
&(\rho+\lambda)\frac{\partial \hat v(x,a)}{\partial x}-\tilde
U'\left(\frac{\partial \hat v(x,a)}{\partial
x}\right)\frac{\partial^2\hat v(x,a)}{\partial x^2}-\frac{\partial
g}{\partial x} = 0.
\end{align*}
Using the equality $\tilde U'(U'(y)) = -y$, the last equation can
be rewritten in terms of $\hat c$:
$$
(\rho+\lambda)U'\left( \hat c(x,a)\right)+ \hat
c(x,a)\frac{\partial^2\hat v(x,a)}{\partial x^2}-\frac{\partial
g}{\partial x} = 0.
$$
Assembling all the pieces together, we obtain the final result
\eqref{auto}.
\end{proof}

\noindent The equation \eqref{auto} is a second-order ODE similar
to equations of theoretical mechanics (second Newton's law), and
it should be solved on the interval $[0,+\infty)$ with the
boundary conditions $Y_0 = x$ and $Y_{\infty} = l(a)$ (which
corresponds to resetting the time to zero after the last trading
date). Solving this equation does not require the auxiliary value
function $\hat v$ but only the original value function $v$, which,
in the case of power utility, can be found from the scaling
relation.

\paragraph{The case of power utility.} In the case of power utility
function $U(x)=K_1 x^\gamma$, the equation \eqref{auto} takes the
form
\begin{align}
\frac{\ud^2 Y_t}{\ud t^2} = \frac{\rho+\lambda}{1-\gamma}c_t -
\frac{1}{K_1 \gamma(1-\gamma)}c^{2-\gamma}_t g'(Y_t),\quad
Y_0=x,\quad Y_\infty = l(a). \label{bvp}
\end{align}
In this case, one can deduce a simple exponential lower bound on
the integrated consumption, corresponding to the solution of
\eqref{bvp} in the case $g\equiv 0$.
\begin{proposition}\label{expbound}
The process $Y$ solution of \eqref{bvp} satisfies
$$
Y_t \geq Y^0_t,
$$
where $Y^0$ is the solution of \eqref{bvp} with $g\equiv 0$, given
explicitly by
\begin{align}
Y^0_t = x - (x-l(a))(1-e^{-\frac{(\rho+\lambda)t}{1-\gamma}}).
\end{align}
\end{proposition}
\noindent The condition $g\equiv 0$ means that the value function
of the investor resets to zero (the investor dies) at a random
future time. In this case it is clear that a rational agent will
consume faster than in the case where more interesting investment
opportunities are available.  The typical shape of optimal
consumption policies is plotted in Figure \ref{cons.fig}.
\begin{proof}
The equation \eqref{bvp} can be rewritten as
$$
\frac{\ud c_t}{\ud t} = -\frac{\rho+\lambda}{1-\gamma}c_t + f(t),
\quad f(t)\geq 0.
$$
>From Gronwall's inequality we then find
\begin{align*}
c_t &\geq c_s e^{-\frac{\rho+\lambda}{1-\gamma}(t-s)}, \\
Y_t &\leq Y_s -
\frac{c_s(1-\gamma)}{\rho+\lambda}(1-e^{-\frac{\rho+\lambda}{1-\gamma}(t-s)}),\quad
\quad t\geq s.
\end{align*}
The terminal condition $Y_\infty = l(a)$ implies
$$
l(a)\leq Y_t - \frac{c_t(1-\gamma)}{\rho+\lambda}.
$$
On the other hand, the solution of the problem without investment
opportunities satisfies
$$
l(a) = Y^0_t - \frac{c^0_t(1-\gamma)}{\rho+\lambda}.
$$
Therefore,
$$
Y_t - \frac{c_t(1-\gamma)}{\rho+\lambda} \geq Y^0_t -
\frac{c^0_t(1-\gamma)}{\rho+\lambda}
$$
and
$$
\frac{\ud}{\ud t}(Y^0_t - Y_t) \leq -\frac{\rho+\lambda}{1-\gamma}
(Y^0_t - Y_t).
$$
Since $Y^0_0 = Y_0 = x$, another application of Gronwall's
inequality shows that $Y^0_t \leq Y_t$ for all $t$.
\end{proof}

\begin{figure}
\centerline{\includegraphics[width=0.55\textwidth]{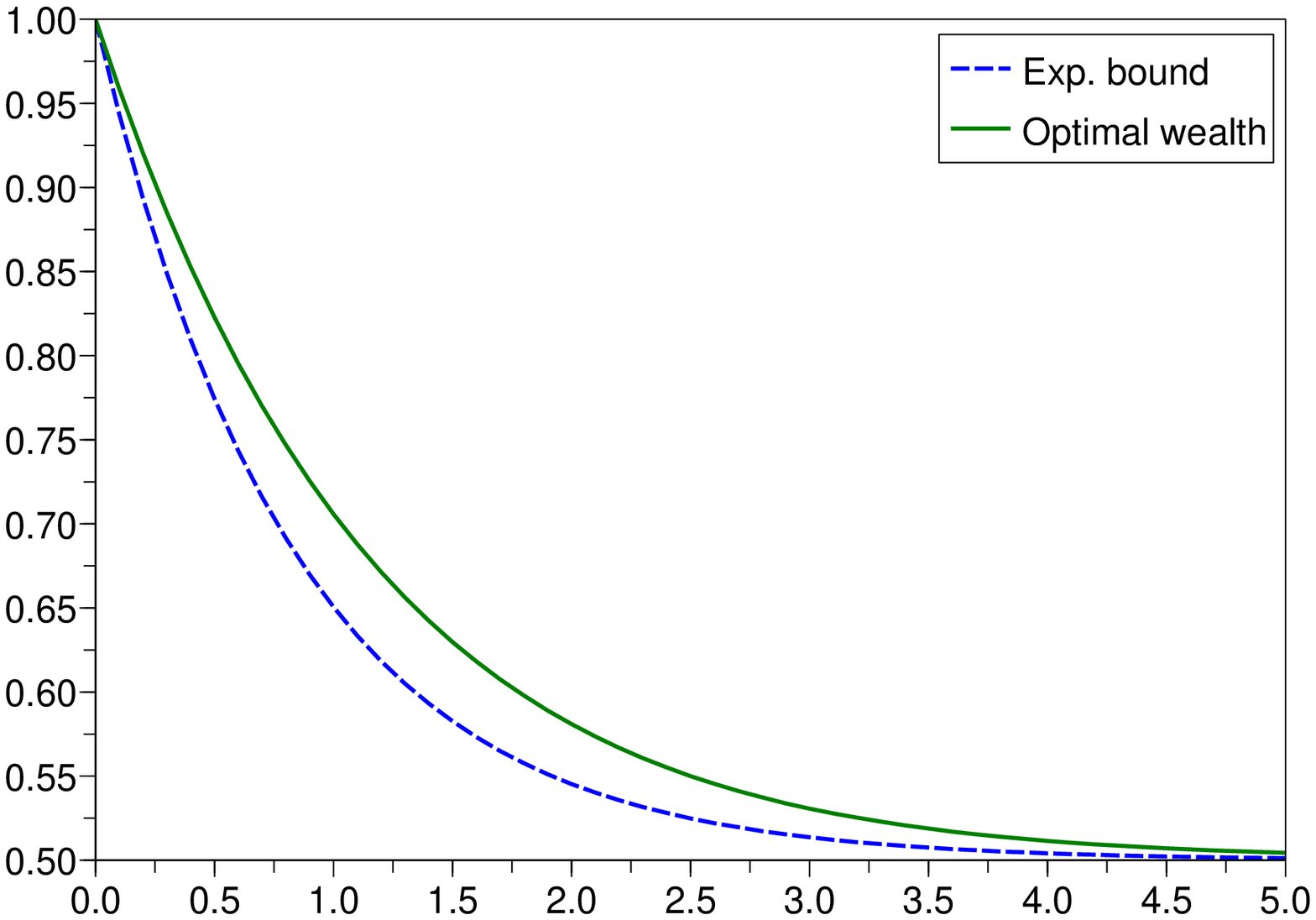}\includegraphics[width=0.55\textwidth]{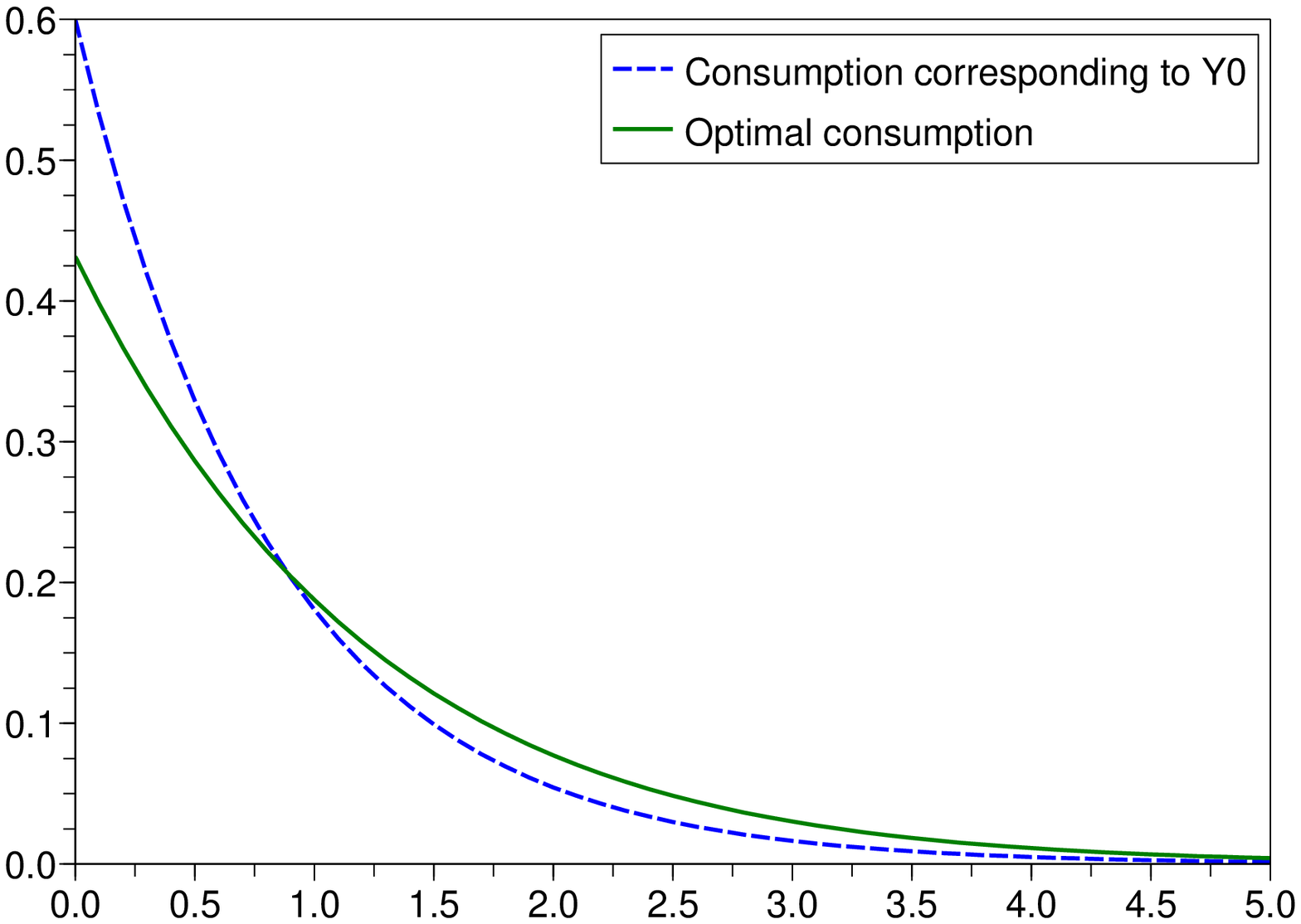}}
\caption{Left: typical profile of the optimal wealth process $Y_t$
and the exponential lower bound given by the proposition
\ref{expbound}. Right: the corresponding consumption strategies.
In the presence of investment opportunities, the agent first
consumes slowly but if the investment opportunity does not appear,
the agent eventually ``gets disappointed'' and starts to consume
fast.} \label{cons.fig}
\end{figure}


\subsubsection{The nonstationary case}

In this case the regularity results for the optimal strategies are
weaker and more difficult to prove.
\begin{proposition} \label{pr:regfacilenonstaz}
Let $a\in A$ and $(t,x) \in \R_+ \times [l(a),+\infty) $. Let
$(\bar c_\cdot, \bar Y_\cdot )$ be the optimal couple for the
auxiliary problem starting at $(t,x)$. If $x=l(a)$, then $\bar
c\equiv 0$, so  $\bar Y \equiv l(a)$. If $x>l(a)$ then $\bar c$ is
continuous, strictly positive and $\lim_{t \to +\infty} \bar
c_t=0$.
\end{proposition}

\begin{proof}
The proof is the same as in the stationary case.
\end{proof}

\noindent Note that, with respect to the stationary case here we
do not have monotonicity of the optimal consumption since the
behavior of $\hat v$ in the time variable is not known.

\noindent Moreover here the limiting property for $\bar Y$ is
proved only under the assumption of twice continuous
differentiability of $U$, as given below.

\noindent As in the stationary case we can deduce an autonomous
equation for the optimal wealth process between two trading dates.
However, since we have weaker regularity results the proof is
different and makes use of the maximum principle.
\begin{proposition}\label{pr:regdiff}
Suppose that $U\in C^2((0,\infty))$ with $U''(x)<0$ for all $x$.
Then the optimal wealth process $Y_s$ between two trading dates is
twice differentiable, it satisfies the second-order ODE
\begin{equation} \label{nonauto}
\frac{\ud^2 Y_s}{\ud s^2} = \frac{\frac{\partial
g(s,Y_s)}{\partial x} - (\rho+\lambda)U'(c_s)}{U''(c_s)},\quad c_s
= -\frac{\ud Y_s}{\ud s}, \quad Y_t=x
\end{equation}
and $\lim_{t \to +\infty} \bar Y_t=l(a)$.
\end{proposition}
\begin{proof}
We cannot differentiate equations \eqref{ipde} and \eqref{genchat}
with respect to $x$ as in the stationary case as we do not know if
$\hat v$ is $C^2$. Then we follow a different approach. We use the
maximum principle contained in Theorem 12 p 234 of~\cite{SS}. Such
theorem concerns problems with endpoint constraints but without
state constraints. Due to the positivity of the consumption, our
auxiliary problem \eqref{vhat} can be easily rephrased
substituting the state constraint $Y_s \ge l(a), \forall s\ge t$
with the endpoint constraint $\lim_{s\to +\infty}Y_s \ge l(a)$. So
we can apply the above quoted theorem that, applied to our case,
states the following:

\medskip

\noindent {\em Assume that $g(\cdot, \cdot)$ and $\frac{\partial
g(\cdot, \cdot)}{\partial x}$ are continuous. Given an optimal
couple $(\bar Y_\cdot , \bar c_\cdot )$  with $\bar c$ continuous
there exists a function $p(\cdot)\in C^1(t,+\infty; \R)$ such
that:
\begin{itemize}
\item $p(\cdot)$ is a solution of the equation
$$
p'(s)= (\rho +\lambda) p(s)- \frac{\partial g(s,\bar
Y_s)}{\partial x};
$$
\item $U'(\bar c_s)=p(s) \; \leftrightarrow \;\bar c_s =I(p(s))$ for every $s\ge
t$; \item $\lim_{T \to + \infty} e^{(\rho + \lambda)(s-T)}p(T)=0$,
for every $t\le s \le T$ (transversality condition).
\end{itemize}
}

\noindent Since we already know (from Proposition
\ref{genuniqueness}) that there exists a unique optimal couple
$(\bar Y_\cdot , \bar c_\cdot )$ and that $\bar c$ is continuous
(see of Proposition \ref{pr:regfacilenonstaz}) the above
statements apply. Then we get that $\bar c_s>0 $ for every $s\ge
t$, that $\bar c$ is everywhere differentiable and that $$
\frac{\ud \bar c_s}{\ud s}=I'(p(s))p'(s)=\frac{1}{U''(\bar
c_s)}\left[(\rho +\lambda)U'(\bar c_s) -  \frac{\partial g(s,\bar
Y_s)}{\partial x} \right]$$ which gives the claim recalling that
$\bar c_s = -\ds \frac{\ud \bar Y_s}{\ud s}$.

\noindent Concerning the limiting property of $\bar Y$ we argue by
contradiction. Let $\lim_{s \to + \infty} \bar Y_s=x_1>l(a)$. We
have then, by the definition of $g$, for every $s\ge t$,
$$
\frac{\partial g(s,\bar Y_s)}{\partial x}\le \frac{\partial
g(s,x_1)}{\partial x} \le \lambda v'(x_1-l(a)) <+\infty.
$$
Then from the costate equation we get that, for $t\le s \le T
<+\infty$
$$
p(s)\le e^{(\rho + \lambda)(s-T)}p(T)+ \int_s^T e^{(\rho +
\lambda)(r-T)}\lambda
 v'(x_1-l(a)) \ud r
$$
$$
\le  e^{(\rho + \lambda)(s-T)}p(T)+ \frac{\lambda}{\rho + \lambda}
 v'(x_1-l(a)) (1-e^{-(\rho + \lambda)T})
$$
Using that $\lim_{T \to + \infty} e^{(\rho + \lambda)(s-T)}p(T)=0$
we get a uniform bound for $p(s)$. This is a contradiction as
$\lim_{s \to + \infty} p(s)=\lim_{s \to + \infty} U'(c_s)=+\infty
$.
\end{proof}

\noindent The equation \eqref{nonauto} is a second-order ODE
similar to equations of theoretical mechanics (second Newton's
law), and it should be solved on the interval $[0,+\infty)$ with
the boundary conditions $Y_0 = x$ and $Y_{\infty} = l(a)$ (which
corresponds to resetting the time to zero after the last trading
date). Solving this equation does not require the auxiliary value
function $\hat v$ but only the original value function $v$, which,
in the case of power utility, can be found from the scaling
relation.

\begin{remark}\label{rm:MPwithoutregularity}
The Maximum Principle used in the above proof holds once we know
that $g(\cdot, \cdot)$ and $\frac{\partial g(\cdot,
\cdot)}{\partial x}$ are continuous. As observed in Remark
\ref{gregularity}, this is true also in cases when the
semiconcavity assumption \ref{h5} may fail (notably in the case of
power utility and in the case of `regular' density). So, also in
such cases the Maximum Principle could be used to get information
about the optimal strategies. Clearly, without knowing the
regularity of the value function $\hat v$ such information would
be much less satisfactory.
\end{remark}

\paragraph{The case of power utility.} In the case of power utility
function, the equation \eqref{nonauto} can again be simplified:
\begin{equation*}
\frac{\ud^2 Y_t}{\ud t^2} = \frac{\rho+\lambda}{1-\gamma}c_t -
\frac{\lambda \vartheta_1 c^{2-\gamma}_t}{K_1 (1-\gamma)} \int
(Y_t + az)^{\gamma-1}p(t,\ud z),\quad Y_0=x,\quad Y_\infty = l(a).
\end{equation*}
Because the second term in the right-hand side is still positive,
the exponential bound of Proposition \ref{expbound} can be
established in exactly the same way as in the stationary case.
Figure \ref{bscons} depicts the optimal wealth process and the
optimal consumption policy for the probability distribution
$p(t,\ud z)$ extracted from the Black-Scholes model with the same
parameter values as in~\cite{pt1}: drift $b=0.4$, volatility
$\sigma=1$, discount factor $\rho=0.2$, intensity $\lambda = 2$
and risk aversion coefficient $\gamma = 0.5$. We see that at least
qualitatively, the consumption profile is similar to the one
observed in the stationary model, with exponential decay. For
comparison, we also plot the wealth and consumption policy for the
stationary model with distribution corresponding to the
Black-Scholes model in $3$ years' time. In this case the agent
consumes at a slower rate than in the nonstationary model. The
explanation is that for the parameter values we chose, 3 years is
a very long time horizon, because all the consumption happens,
essentially, during the first 2 years after trading. During this
period (first 2 years) the stationary model offers better
investment opportunities, which explains the slower consumption
rate.

\begin{figure}
\centerline{\includegraphics[width=0.6\textwidth]{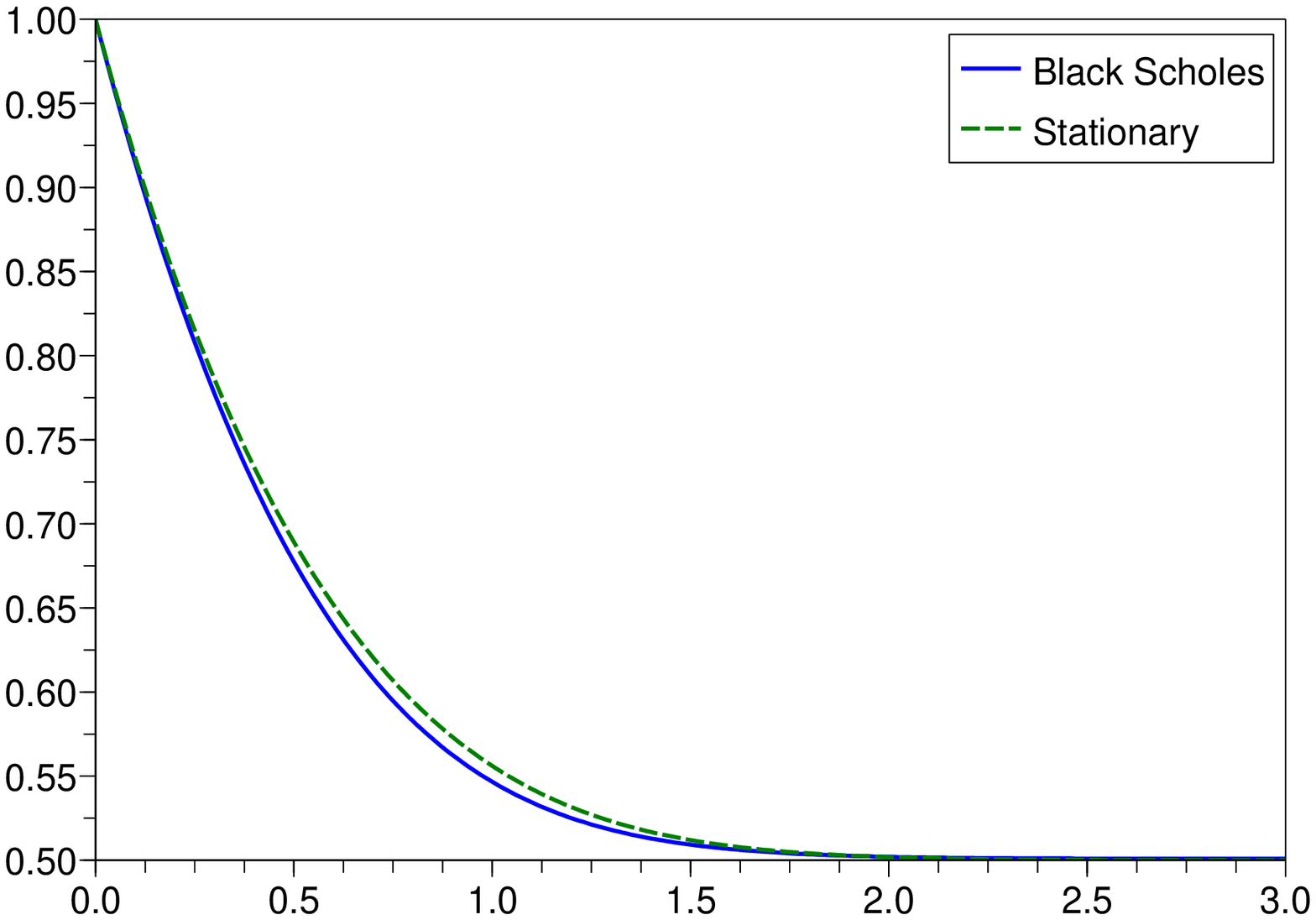}\hspace*{-1cm}\includegraphics[width=0.6\textwidth]{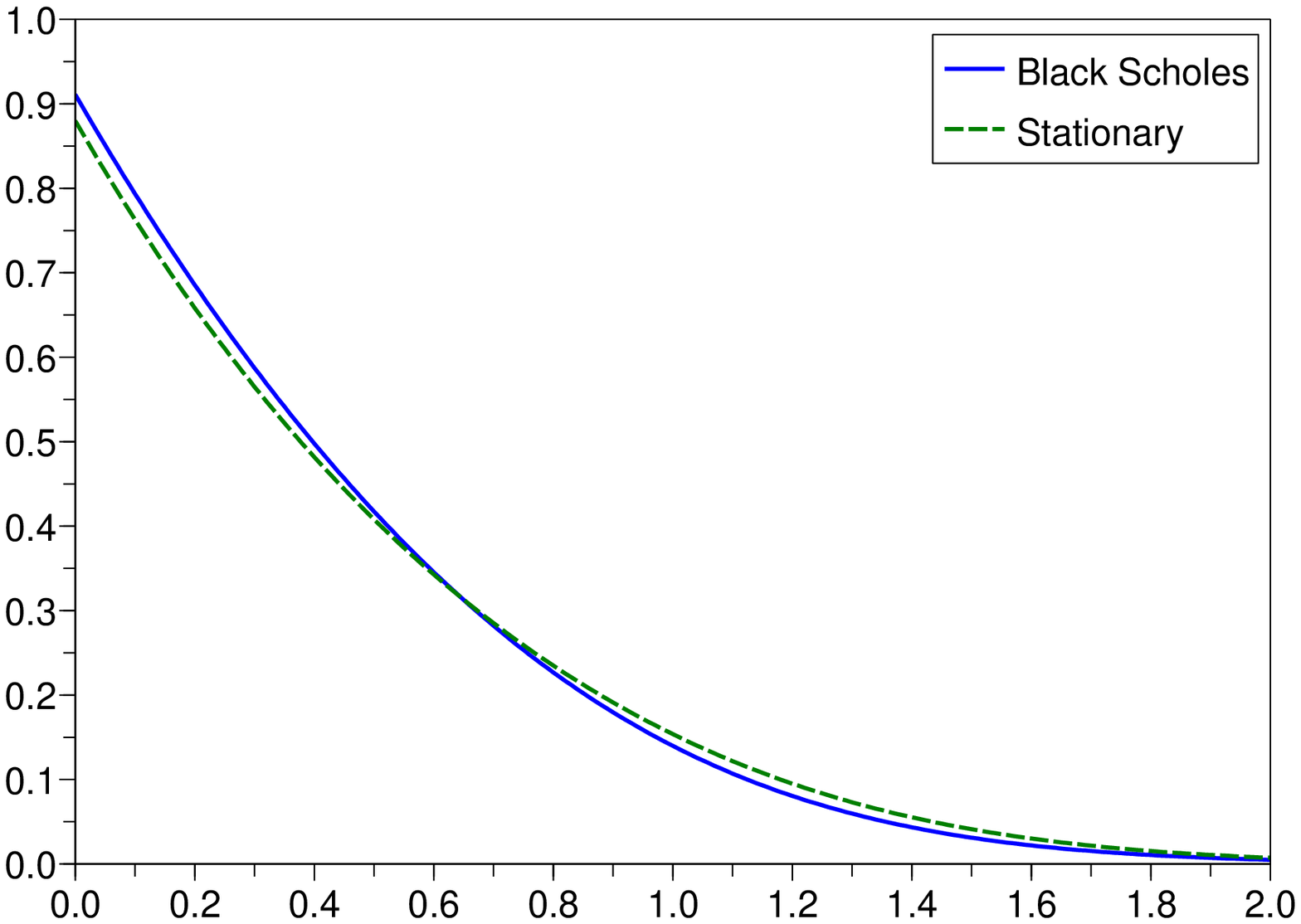}}
\caption{Optimal wealth (left) and consumption policy (right) for
the probability distribution extracted from the Black-Scholes
model (solid line) and from the stationary model having the same
distribution as the Black-Scholes model in $3$ years' time (dashed
line).} \label{bscons}
\end{figure}


\appendix

\section{Appendix~: Technical proofs}

\paragraph{Proof of Proposition \ref{vstrincr}.}
We suppose by contradiction that $v$ is not strictly increasing on
$\R_+$ This means that it is definitely constant on $\R_+$ from a
certain $x$ on, since $v$ is concave. Then we fix $\bar x \in
\R_+$ such that $v(x)=B \in \R_+$, for all $x \geq \bar x$. Take
$\epsilon >0$ and a pair $(\alpha^\epsilon,c^\epsilon)$
$\epsilon$-optimal at $\bar x$. This means that
$(\alpha^\epsilon,c^\epsilon) \in \A(\bar x)$, i.e.
$$
X_k^{\bar x}=\bar x-\int_0^{\tau_k}c^\epsilon_t\ud
t+\sum_{i=1}^k\alpha^\epsilon_iZ_i\geq 0, \quad \forall k \geq
1,\quad X_0^{\bar x}=\bar x,
$$
and
$$
B=v(\bar x)<\esp{\int_0^{+\infty}e^{-\rho t}U(c_t^\epsilon)\ud
t}+\epsilon.
$$
Now we choose $\tilde x>\bar x+1$. Then we have $v(\tilde
x)=v(\bar x)= B$. We consider the control policy
$(\alpha^\epsilon,\tilde c)$, where $\tilde
c_t=c_t^\epsilon+\mathbb I_{[0,1]}(t)$, for all $t \geq 0$. Hence
given $\tilde x>0$, we have for every $k \geq 1$,
\begin{align*}
X_k^{\tilde x}& =\tilde x-\int_0^{\tau_k}\tilde c_t\ud
t+\sum_{i=1}^k\alpha^\epsilon_iZ_i=\tilde
x-\int_0^{\tau_k}c_t^\epsilon\ud
t-(1\wedge \tau_k)+\sum_{i=1}^k\alpha^\epsilon_iZ_i\\
& >\bar x-\int_0^{\tau_k}c_t^\epsilon\ud
t+\sum_{i=1}^k\alpha^\epsilon_iZ_i \geq 0,
\end{align*}
with $X_0^{\tilde x}=\tilde x$, so $(\alpha^\epsilon,\tilde c) \in
\A(\tilde x)$. Moreover we have:
\begin{align*}
v(\tilde x) \geq \mathbb E\bigg[\int_0^{+\infty}&e^{-\rho
t}U(\tilde c_t)\ud t\bigg]=\esp{\int_0^1e^{-\rho
t}U(c_t^\epsilon+1)\ud
t}+\esp{\int_1^{+\infty}e^{-\rho t}U(c_t^\epsilon)\ud t}\\
&>\esp{\int_0^1e^{-\rho t}U(c_t^\epsilon)\ud
t}+\esp{\int_1^{+\infty}e^{-\rho t}U(c_t^\epsilon)\ud t}=v(\bar
x)=B,
\end{align*}
since $U$ is strictly increasing. But this is not possible, since
we have assumed $v$ constant from $\bar x$ on. Hence the statement
is proved.
\begin{flushright}
$\square$
\end{flushright}

\paragraph{Proof of Proposition \ref{gconcave}.}
\begin{enumerate}
\item [(i)] The continuity comes from condition d) of Assumption \ref{H}. If d) does not
hold, measurability follows from condition b) of Assumption
\ref{H}.
\item [(ii)] The function $g$ is strictly increasing in $x \in
[l(a),+\infty)$ since $v$ is strictly increasing by Proposition
\ref{vstrincr}.
\item [(iii)] This property is a
direct consequence of concavity of $v$. Indeed, given $t \geq 0$,
consider $(x_\eta,a_\eta)=(\eta x_1+(1-\eta)x_2,\eta
a_1+(1-\eta)a_2)$, with $\eta \in (0,1)$, $x_1 \geq l(a_1),x_2
\geq l(a_2)$. First of all, $x_\eta \geq l(a_\eta)$ thanks to the
convexity of the function $l$. Since $v$ is concave, we have for
every $t \geq 0$:
\begin{align*}
g(t,x_\eta,a_\eta)&=\lambda\int v\left(\eta x_1+(1-\eta)x_2+\eta
a_1z+(1-\eta)a_2z\right)p(t,\ud z)\\
&\geq \lambda \eta\int v\left(x_1+a_1z\right)p(t,\ud
z)+\lambda(1-\eta)\int v\left(x_2+a_2z\right)p(t,\ud z)\\
&=\lambda \eta g(t,x_1,a_1)+\lambda (1-\eta)g(t,x_2,a_2).
\end{align*}
\noindent This provides the result.
\end{enumerate}
\begin{flushright}
$\square$
\end{flushright}

\paragraph{Proof of Proposition \ref{genuniqueness}.} In order to prove Proposition \ref{genuniqueness},
we need the following preliminary result:
\begin{lemma} \label{identityrel}
Let $\hat v$ be the value function given in \eqref{vhat}. Fix $a
\in A$. Assume the followings:
\begin{itemize}
\item[(i)] $\hat v(\cdot,\cdot,a) \in
C^1\left(\R_+\times(l(a),+\infty)\right)$;
\item[(ii)] $\ds \frac{\partial \hat v(t,l(a)^+,a)}{\partial x}=+\infty$, for
every $t \in \R_+$;
\item[(iii)] $\hat v$ is a classical solution of the HJ
equation \eqref{ipde} satisfying the growth condition
\eqref{hatvgrowth} with representation \eqref{generalboundary} on
the boundary.
\end{itemize}
Given $x \in [l(a),+\infty)$ and $t \ge 0$, for every couple
$(c,Y)$ admissible at $(t,x)$ for $s\ge t$, we have the following
identity: for $T>t$
\begin{equation}\label{identity}
\begin{split}
& e^{-(\rho+\lambda)T}\hat
v\left(T,Y_T,a\right)-e^{-(\rho+\lambda)t}\hat
v(t,x,a)=-\int_t^Te^{-(\rho+\lambda)s}\left[U(c_s)+g(s,Y_s,a)\right]\ud
s\\
& \quad \quad+\int_t^Te^{-(\rho+\lambda)s}\bigg[U(c_s)
-c_s\frac{\partial \hat v(s,Y_s,a)}{\partial x}-\tilde
U\left(\frac{\partial \hat v(s,Y_s,a)}{\partial x}\right)\bigg]\ud
s,
\end{split}
\end{equation}
with the agreement that
$$\frac{\partial \hat v(t,l(a),a)}{\partial
x}=\frac{\partial \hat v(t,l(a)^+,a)}{\partial x}=+\infty,\
\mbox{so\ that}\ \tilde U\left(\frac{\partial \hat
v(s,l(a),a)}{\partial x}\right)=0.
$$
If $T$ goes to $+\infty$
\begin{equation}\label{identityinf}
\begin{split}
&\hat
v(t,x,a)=\int_t^{+\infty}e^{-(\rho+\lambda)(s-t)}\bigg[U(c_s)+g(s,Y_s,a)\bigg]\ud s\\
&\quad -
\int_t^{+\infty}e^{-(\rho+\lambda)(s-t)}\bigg[U(c_s)-c_s\frac{\partial
\hat v(s,Y_s,a)}{\partial x}-\tilde U\left(\frac{\partial \hat
v(s,Y_s,a)}{\partial x}\right)\bigg]\ud s.
\end{split}
\end{equation}
Furthermore, an admissible couple $(c,Y)$ is optimal  at $(t,x)$
if and only if
\begin{equation*}
\tilde U\bigg(\frac{\partial \hat v(s,Y_s,a)}{\partial
x}\bigg)=U(c_s)- c_s\frac{\partial \hat v(s,Y_s,a)}{\partial
x},\quad {\rm for\ a.e.}\ s \ge t
\end{equation*}
such that $Y_s > l(a)$ and $c_s=0$ otherwise.
\end{lemma}

\begin{proof}
Let $(c,Y)$ be an admissible couple for the auxiliary problem such
that $Y_s>l(a)$, for every $s\ge t$. By applying standard
differential calculus to $e^{-(\rho+\lambda)s}\hat v(s,Y_s,a)$
between $s=t$ and $s=T$, we have:
\begin{align*}
&e^{-(\rho+\lambda)T}\hat v(T,Y_T,a)-e^{-(\rho+\lambda)t}\hat v(t,x,a)\\
& \ =\int_t^Te^{-(\rho+\lambda)s}\left[\frac{\partial \hat
v(s,Y_s,a)}{\partial t}-(\rho+\lambda)\hat
v(s,Y_s,a)-c_s\frac{\partial \hat v(s,Y_s,a)}{\partial
x}\right]\ud
s\\
& \quad =\int_t^Te^{-(\rho+\lambda)s}\left[-\tilde
U\left(\frac{\partial \hat v(s,Y_s,a)}{\partial
x}\right)-g(s,Y_s,a)-c_s\frac{\partial \hat v(s,Y_s,a)}{\partial
x}\right]\ud s,
\end{align*}
where in the last equation we have used the fact that $\hat v$
satisfies \eqref{ipde}. This can be easily rewritten as
\eqref{identity} by adding and subtracting $U(c_s)$ in the
integrand. Now, from the growth condition \eqref{hatvgrowth} and
since $\hat v$ is nondecreasing in $x$, we have
$$
0 \leq \hat v(T,Y_T,a) \leq \hat v(T,x,a) \leq K(e^{bT}x)^\gamma
\quad {\rm a.s.}
$$
from which we deduce by Lemma 4.2 of~\cite{pt} that
$$
\lim_{T \to +\infty}e^{-(\rho+\lambda)T}\hat v\left(T,
Y_T,a\right)=0,\quad {\rm a.s.}
$$
Hence, by sending $T$ to infinity, we can easily derive the
relation \eqref{identityinf}. Let $(c,Y)$ be an admissible couple
such that $Y_{T_0}=l(a)$, for a $T_0<+\infty$. Assume that $T_0$
is the first time when this happens. Then $Y_s=l(a)$, and $c_s=0$
for every $s \ge T_0$. Then for $T<T_0$ we get \eqref{identity} as
before. Calling
\begin{equation*}
\begin{split}
I_T:=-\int_t^Te^{-(\rho+\lambda)s}\left[U(c_s) -c_s\frac{\partial
\hat v(s,Y_s,a)}{\partial x}-\tilde U\left(\frac{\partial \hat
v(s,Y_s,a)}{\partial x}\right)\right]\ud s,
\end{split}
\end{equation*}
we have that $I_T$ is increasing and from \eqref{identity} that
there exists its limit for $T\nearrow T_0$ given by:
$$
-e^{-(\rho+\lambda)T_0}\hat
v\left(T_0,l(a),a\right)+e^{-(\rho+\lambda)t}\hat
v(t,x,a)-\int_t^{T_0}e^{-(\rho+\lambda)(s-t)}\left[U(c_s)+g(s,Y_s,a)\right]\ud
s.
$$
>From the positivity of the integrand in $I_T$, we then get that
identity \eqref{identity} also holds in $T_0$. For $T>T_0$ we can
easily derive \eqref{identity} using the fact that the couple
$(c,Y)$ is constant after $T_0$ and that (ii) holds. Now, let us
focus on the last statement. Let $(c,Y)$ be an admissible couple
at $(t,x)$. Then $(c,Y)$ is optimal at $(t,x)$ if and only if in
\eqref{identityinf} we have
\begin{equation*}
\hat
v(t,x,a)=\int_t^{+\infty}e^{-(\rho+\lambda)(s-t)}\left[U(c_s)+g(s,Y_s,a)\right]\ud
s.
\end{equation*}
When $Y_s > l(a)$, for $s \ge t$, this is clearly equivalent to
\begin{equation*}
\int_t^{+\infty}e^{-(\rho+\lambda)(s-t)}\left[U(c_s)-c_s\frac{\partial
\hat v(s,Y_s,a)}{\partial x}-\tilde U\left(\frac{\partial \hat
v(s,Y_s,a)}{\partial x}\right)\right]\ud s=0,
\end{equation*}
i.e.
\begin{equation}\label{urel}
\tilde U\left(\frac{\partial \hat v(s,Y_s,a)}{\partial
x}\right)=U(c_s)- c_s\frac{\partial \hat v(s,Y_s,a)}{\partial
x},\quad {\rm for}\ {\rm a.e.}\ s \ge t.
\end{equation}
When $Y_s > l(a)$ on $(t,T)$, we have \eqref{urel} on $(t,T)$ and
$c_s=0$ on $[T,+\infty)$.
\end{proof}

\noindent Now we come to the proof of the Proposition
\ref{genuniqueness}. First we observe that, thanks to Proposition
\ref{smoothness} the assumptions (i)-(ii)-(iii) of the previous
Lemma \ref{identityrel} hold. So fix $(t,x,a) \in \D$. First we
prove the existence of a solution $\bar Y$ of the problem
\eqref{cpgen}. The dynamics of the system is the function $-\hat
c(\cdot,\cdot,a):\R_+ \times (l(a),+\infty) \to (0,+\infty)$, with
\eqref{genchat}, that is well-defined and continuous as
composition of continuous functions on $\R_+ \times
(l(a),+\infty)$. We note that hypothesis (ii) of Lemma
\ref{identityrel} implies $\hat c(t, l(a)^+,a)=0$, for every $t
\ge 0$. Hence, we can extend the function $\hat c(\cdot,\cdot,a)$
to a continuous function on $\R_+ \times (-\infty,+\infty)$ such
that $\hat c=0$ on $\R_+ \times (-\infty,l(a)]$. Now the Peano's
Theorem guarantees the existence of a local solution $\bar
Y_\cdot$ of \eqref{cpgen}. We prove that $(s,\bar Y_s,a) \in \D$
for every $s \geq t$, i.e. that
\begin{equation} \label{uno2}
\bar Y_s \geq l(a),\quad \mbox{for}\ s \geq t.
\end{equation}
If $x=l(a)$, we already know that $\hat c(s,l(a)^+,a)=0$, for $s
\geq t$, given $t$, so that
$\bar Y_s=l(a)$, for all $s \geq t$.\\
Now we suppose $x>l(a)$. Since $-\hat c(s,y,a)< 0$, for each
$(s,y) \in [t,+\infty) \times (l(a),+\infty)$, the solution $\bar
Y$ is strictly decreasing on the maximal interval that we denote
by $(t,T)$, with $T>0$. Suppose that there exists an instant
$t<t^\prime<T$ such that $\bar Y_{t^\prime}<l(a)$. We have that
$\ud \bar Y_{t^\prime}=0$. In particular this means that there
exists an interval $[t_0,t_1]\subset (t,T)$ with $\bar
Y_{t_0}=l(a)$ and $\bar Y_{t_1}<l(a)$ such that for all $s \in
(t_0,t_1]$, $\hat Y_s<l(a)$ with $\ud \bar Y_s(t,x,a)=0$, that it
is not possible. This proves the claim
\eqref{uno2}, for any $x \geq l(a)$ and that $T=+\infty$.\\
Now call $\bar c_s=\hat c(s,\bar Y_s,a)$ as in \eqref{gencbar}.
Then the couple $(\bar c,\bar Y)$ is admissible since $\bar
c_s\geq 0$, for every $s \ge t$ and $\bar Y_s \ge l(a)$, for $s
\ge t$. Moreover
$$
\tilde U\left(\frac{\partial \hat v}{\partial x}(s,\bar
Y_s,a)\right)=U(c_s)-\bar c_s\frac{\partial \hat v}{\partial
x}(s,\bar Y_s,a),\quad \mbox{for}\ {\rm a.e.}\ s\ge t,
$$
so the couple $(\bar c,\bar Y)$ is optimal at $(t,x)$ thanks to
Lemma \ref{identityrel}. Hence the existence of an optimal couple
for the auxiliary problem is proved. \\ Now we prove the
uniqueness.
Fix $a \in \A$, $x \geq l(a)$ and $t \geq 0$. Let $\bar c_1$,
$\bar c_2$ be optimal controls at $x$. Then for $i=1,2$
\begin{align*}
\hat v(t,x,a)&=\int_t^{+ \infty}e^{-(\rho +
\lambda)(s-t)}\left[U(\bar c_i(s))+g(s,\bar Y_s^{t,x}(\bar c_i),a)\right]\ud s\\
&=\int_0^{+\infty}e^{-(\rho + \lambda)s}\left[U(\bar
c_i(s))+g(s+t,\bar Y_s^x(\bar c_i),a)\right]\ud s,
\end{align*}
where for every $c \in \C_a(x)$, $Y_s^x(c)=x-\int_0^sc(u)\ud u,\ s
\geq 0$. Since the function $U$ is strictly concave, we have by
setting $c_\eta=\eta \bar c_1+(1-\eta)\bar c_2$, with $\eta \in
(0,1)$,
$$
U(c_\eta(s))=U\left(\eta \bar c_1(s)+(1-\eta)\bar
c_2(s)\right)>\eta_1U(\bar c_1(s))+(1-\eta)U(\bar c_1(s)), \quad
s\geq 0.
$$
Moreover, since $\bar Y_s^x(c_\eta)=\eta \bar Y_s^x(\bar
c_1)+(1-\eta)\bar Y_s^x(\bar c_2)$, for all $s \geq 0$ and $g$ is
concave in the second variable, we have
$$
g(s+t,\bar Y_s^x(c_\eta),a) \geq \eta g(s+t,\bar Y_s^x(\bar
c_1),a)+(1-\eta)g(s+t,\bar Y_s^x(\bar c_2),a),\quad  \forall s
\geq 0.
$$
Then
$$
\hat v(t,x,a)<\int_0^{+\infty}e^{-(\rho +
\lambda)s}\left[U(c_\eta(s))+g(s+t,\bar Y_s^x(c_\eta),a)\right]\ud
s,
$$
that implies the uniqueness of the control of the auxiliary
problem.
\begin{flushright}
$\square$
\end{flushright}

\end{document}